\documentclass{amsart}
\usepackage{amsmath}
\usepackage{amsthm}
\usepackage{amssymb}
\usepackage{pstricks}
\usepackage{pst-vue3d}
\usepackage{multido}

\theoremstyle{plain}
  \newtheorem{theorem}{\bf Theorem}[section]
  \newtheorem{proposition}[theorem]{\bf Proposition}
  \newtheorem{lemma}[theorem]{\bf Lemma}
  \newtheorem{corollary}[theorem]{\bf Corollary}
\newtheorem{conjecture}[theorem]{\bf Conjecture}
\theoremstyle{remark}
  \newtheorem{remark}[theorem]{\bf Remark}
\theoremstyle{remarks}

\DeclareMathOperator{\coker}{coker}

\begin{document}
\title[Nonlinear Brascamp--Lieb inequalities]{Some nonlinear Brascamp--Lieb inequalities and
applications to harmonic analysis} \keywords{Brascamp--Lieb
inequalities, induction-on-scales, Fourier extension estimates}
\subjclass[2000]{44A12; 42B10; 44A35}
\author{Jonathan Bennett}
\author{Neal Bez}
\thanks{Both authors were supported by EPSRC
grant EP/E022340/1.}
\address{Jonathan Bennett and Neal Bez, School of Mathematics, The Watson Building, University of Birmingham, Edgbaston,
Birmingham, B15 2TT, England.} \email{J.Bennett@bham.ac.uk}
\email{N.Bez@bham.ac.uk}
\date{29th June 2010}
\begin{abstract}
We use the method of induction-on-scales to prove certain
diffeomorphism invariant nonlinear Brascamp--Lieb inequalities. We
provide applications to multilinear convolution inequalities and the
restriction theory for the Fourier transform, extending to higher
dimensions recent work of Bejenaru--Herr--Tataru and
Bennett--Carbery--Wright.
\end{abstract}
\maketitle
\section{Introduction}
The purpose of this paper is to obtain nonlinear generalisations of
certain Brascamp--Lieb inequalities and apply them to some
well-known problems in euclidean harmonic analysis. Our particular
approach to such inequalities is by induction-on-scales, and builds
on the recent work of Bejenaru, Herr and Tataru \cite{BHT}.

The Brascamp--Lieb inequalities simultaneously generalise important
classical inequalities such as the multilinear H\"older, sharp Young
convolution and Loomis--Whitney inequalities. They may be formulated
as follows. Suppose $m\geq 2$ and $d, d_1,\hdots,d_m$ are positive
integers, and for each $1\leq j\leq m$,
$B_j:\mathbb{R}^{d}\rightarrow\mathbb{R}^{d_j}$ is a linear
surjection and $p_j\in [0,1]$. The Brascamp--Lieb inequality
associated with these objects takes the form
\begin{equation}\label{bl}
\int_{\mathbb{R}^d}\prod_{j=1}^m (f_j\circ B_j)^{p_j}\leq C\prod_{j=1}^m\left(\int_{\mathbb{R}^{d_j}}f_j\right)^{p_j}
\end{equation}
for all nonnegative $f_j\in L^1(\mathbb{R}^{d_j})$, $1\leq j\leq m$.
Here $C$ denotes a constant depending on the datum
$(\mathbf{B},\mathbf{p}):=((B_j),(p_j))$, which at this level of
generality may of course be infinite.
For nonnegative functions $f_j\in L^1(\mathbb{R}^{d_j})$ satisfying
$0<\int f_j<\infty$, we define the quantity
$$
\mbox{BL}(\mathbf{B},\mathbf{p};\mathbf{f})=\frac{\int_{\mathbb{R}^d}\prod_{j=1}^m
(f_j\circ
B_j)^{p_j}}{\prod_{j=1}^m\left(\int_{\mathbb{R}^{d_j}}f_j\right)^{p_j}},$$
where $\mathbf{f}:=(f_j)$. We may then define the Brascamp--Lieb
constant $0<\mbox{BL}(\mathbf{B},\mathbf{p})\leq\infty$ to be the
supremum of $\mbox{BL}(\mathbf{B},\mathbf{p};\mathbf{f})$ over all
such inputs $\mathbf{f}$. The quantity
$\mbox{BL}(\mathbf{B},\mathbf{p})$ is of course the smallest
$0<C\leq\infty$ for which \eqref{bl} holds. It should be noted here
that there is a natural equivalence relation on Brascamp--Lieb data,
where $(\mathbf{B},\mathbf{p})\sim (\mathbf{B'},\mathbf{p'})$ if
$\mathbf{p}=\mathbf{p'}$ and there exist invertible linear
transformations $C: \mathbb{R}^d \rightarrow \mathbb{R}^{d}$ and
$C_j:\mathbb{R}^{d_j} \rightarrow \mathbb{R}^{d_j}$ such that $B'_j
= C_j^{-1} B_j C$ for all $j$; we refer to $C$ and $C_j$ as the
intertwining transformations. In this case, simple changes of
variables show that
$$
\mbox{BL}(\mathbf{B}',\mathbf{p}')=\frac{\prod_{j=1}^m |\det C_j|^{p_j}}{|\det C|}
\mbox{BL}(\mathbf{B},\mathbf{p}),
$$
and thus $\mbox{BL}(\mathbf{B},\mathbf{p})<\infty$ if and only if
$\mbox{BL}(\mathbf{B}',\mathbf{p}')<\infty$. This terminology is
taken from \cite{BCCT}.

The generality of this setup of course raises questions, many of
which have been addressed in the literature. In \cite{L} Lieb showed
that the supremum above is exhausted by centred gaussian inputs,
prompting further investigation into issues including the finiteness
of $\mbox{BL}(\mathbf{B},\mathbf{p})$ and the
extremisability/gaussian-extremisability of
$\mbox{BL}(\mathbf{B},\mathbf{p};\mathbf{f})$. A fuller description
of the literature is not appropriate for the purposes of this paper.
The reader is referred to the survey article \cite{Barthe2} and the
references there.

A large number of problems in harmonic analysis require nonlinear
versions of inequalities belonging to this family; see \cite{BHHT},
\cite{BHT}, \cite{BCW}, \cite{Gressman}, \cite{Stovall}, and
\cite{TW} for instance. The generalisations we seek here are local
in nature, and amount to allowing the maps $B_j$ to be nonlinear
submersions in a neighbourhood of a point $x_0\in\mathbb{R}^d$, and
then looking for a neighbourhood $U$ of $x_0$ such that if $\psi$ is
a cutoff function supported in $U$, there exists a constant $C>0$
for which
\begin{equation}\label{aspiration}
\int_{\mathbb{R}^d}\prod_{j=1}^m f_j(B_j(x))^{p_j}\psi(x)\,\mathrm{d}x\leq C
\prod_{j=1}^m\left(\int_{\mathbb{R}^{d_j}}f_j\right)^{p_j}
\end{equation}
for all nonnegative $f_j\in L^{1}(\mathbb{R}^{d_j})$, $1\leq j\leq
m$. The applications of such inequalities invariably require more
quantitative statements involving the sizes of the neighbourhood $U$
and constant $C$, and also the nature of any
smoothness/non-degeneracy conditions imposed on the nonlinear maps
$(B_j)$.

Notice that if $d_j = d$ for each $j$, then the nonlinear $B_j$ are
of course local diffeomorphisms. In this situation necessarily $p_1
+ \cdots + p_m = 1$ and \eqref{aspiration} follows from the
$m$-linear H\"older inequality. Similar considerations allow to
reduce matters to the case where $d_j<d$ for all $j$.

It is perhaps reasonable to expect to obtain an inequality of the
form \eqref{aspiration} for smooth nonlinear maps $(B_j)$ and
exponents $(p_j)$ for which $\mbox{BL}((\mathrm{d}B_j(x_0)),(p_j)) <
\infty$. Here $\mathrm{d}B_j(x_0)$ denotes the derivative map of
$B_j$ at $x_0$. However, the techniques that we employ in this paper
appear to require additional structural hypotheses on the maps
$\mathrm{d}B_j(x_0)$, and so instead we seek to identify a natural
class
$$
\mathcal{C} \subseteq \{ (\mathbf{B},\mathbf{p}) : \text{each $B_j$ is \emph{linear} and $\mbox{BL}(\mathbf{B},\mathbf{p}) < \infty$} \}
$$
such that \eqref{aspiration} holds for nonlinear $(B_j)$ with
$((\mathrm{d}B_j(x_0)),(p_j)) \in \mathcal{C}$. As will become clear
in Section \ref{section:C0}, a natural choice for consideration is
\begin{equation} \label{e:C0}
\mathcal{C}= \bigg\{ (\mathbf{B},\mathbf{p}) : \bigoplus_{j=1}^m \ker B_j = \mathbb{R}^d, \;\;  p_1=\cdots=p_m=\tfrac{1}{m-1} \bigg\}.
\end{equation}
This class contains the classical Loomis--Whitney datum \cite{LW},
whereby $m=d$, $d_j=d-1$, $p_j=1/(d-1)$ and
$B_j(x_1,\hdots,x_d)=(x_1,\hdots,\widehat{x_j},\hdots,x_d)$ for all
$1\leq j\leq d$. Here $\;\widehat{\;}\;$ denotes omission.

The purpose of this paper is two-fold. Firstly, we establish an
inequality of the form \eqref{aspiration} whenever
$((\mathrm{d}B_j(x_0)),(p_j)) \in \mathcal{C}$, where $\mathcal{C}$
is defined in \eqref{e:C0}. Secondly, we use these inequalities to
deduce certain sharp multilinear convolution estimates, which in
turn yield progress on the multilinear restriction conjecture for
the Fourier transform. These applications can be found in Section
\ref{section:applications}.

Before stating our nonlinear Brascamp--Lieb inequalities, it is
important that we discuss further the class $\mathcal{C}$ given in
\eqref{e:C0}. Notice that the transversality hypothesis
\begin{equation} \label{e:directsum}
  \bigoplus_{j=1}^m \ker B_j = \mathbb{R}^d
\end{equation}
is preserved under the equivalence relation on Brascamp--Lieb data;
that is, it is invariant under $B_j\mapsto C_j^{-1}B_jC$ for
invertible linear transformations
$C:\mathbb{R}^d\rightarrow\mathbb{R}^d$ and
$C_j:\mathbb{R}^{d_j}\rightarrow\mathbb{R}^{d_j}$. By choosing
appropriate intertwining transformations $C$ and $C_j$, an
elementary calculation shows that if
$(\mathbf{B},\mathbf{p})\in\mathcal{C}$ then
$(\mathbf{B},\mathbf{p})\sim (\mathbf{\Pi},\mathbf{p})$, where
$\mathbf{\Pi}=(\Pi_j)_{j=1}^m$ are certain coordinate projections.
In order to define $\Pi_j$ we let $\mathcal{K}_j \subseteq
\{1,\ldots,d\}$ be given by
$$
\mathcal{K}_j = \{d_1' + \cdots + d_{j-1}' + 1, \ldots, d_1' + \cdots + d_{j-1}' + d_j'\},
$$
where $d_j'=d-d_j$ denotes the dimension of the kernel of $B_j$, so
that $\mathcal{K}_1,\ldots,\mathcal{K}_m$ form a partition of
$\{1,\ldots,d\}$. Then we let $\Pi_j : \mathbb{R}^d \rightarrow
\mathbb{R}^{d_j}$ be given by
\begin{equation} \label{e:Pijdefn}
\Pi_j(x) = (x_k)_{k \in \mathcal{K}_j^c}.
\end{equation}
\begin{proposition} \label{p:Finnerorth} \cite{Finner} If $\emph{\textbf{p}} =
(\tfrac{1}{m-1},\ldots,\tfrac{1}{m-1})$ then $\emph{\mbox{BL}}(\mathbf{\Pi},\mathbf{p})=1$, and thus
\begin{equation} \label{e:Finnerorth}
\int_{\mathbb{R}^d} \prod_{j=1}^m f_j(\Pi_jx)^{\frac{1}{m-1}}\,\mathrm{d}x \leq \prod_{j=1}^m \bigg( \int_{\mathbb{R}^{d_j}} f_j \bigg)^{\frac{1}{m-1}}
\end{equation}
holds for all nonnegative $f_j \in L^1(\mathbb{R}^{d_j}), 1 \leq j
\leq m$.
\end{proposition}
Proposition \ref{p:Finnerorth} follows from work of Finner
\cite{Finner} where a stronger result was established for
$\mathbf{\Pi}$ consisting of more general coordinate projections and
in the broader setting of product measure spaces. In particular,
this includes the discrete inequality
\begin{equation} \label{e:discreteFinner}
\sum_{n \in \mathbb{N}^{d}} \prod_{j=1}^m f_j(\Pi_jn)^{\frac{1}{m-1}}
\leq \prod_{j=1}^m \bigg( \sum_{\ell \in \mathbb{N}^{d_j}} f_j(\ell) \bigg)^{\frac{1}{m-1}}
\end{equation}
which holds for all nonnegative $f_j \in \ell^1(\mathbb{N}^{d_j}), 1
\leq j \leq m$. We mention this case specifically as it will be
important later in the paper.

We remark that \eqref{e:Finnerorth} is a generalisation of the
classical Loomis--Whitney inequality \cite{LW} whereby $m=d$ and
$\mathcal{K}_j=\{j\}$ for $1\leq j\leq d$.

In order for $\mbox{BL}(\mathbf{\Pi},\mathbf{p})$ to be finite it is
necessary that $\textbf{p} =
(\tfrac{1}{m-1},\ldots,\tfrac{1}{m-1})$, and this follows by a
straightforward scaling argument.

The standard proof of Proposition \ref{p:Finnerorth} proceeds via
the multilinear H\"older inequality and induction (see
\cite{Finner}). This proof and, to the best of our knowledge, other
established proofs of Proposition \ref{p:Finnerorth} rely heavily on
the linearity of the $\Pi_j$ and break down completely in the
nonlinear setting.

Since we would like to state our main theorem regarding nonlinear
$B_j$ in a diffeomorphism-invariant way, it is appropriate that we
first formulate an affine-invariant version of Proposition
\ref{p:Finnerorth}. In order to state this it is natural to use
language from exterior algebra; the relevant concepts and
terminology can be found in standard texts such as \cite{Darling}.
In particular, $\Lambda^n(\mathbb{R}^d)$ will denote the $n$th
exterior algebra of $\mathbb{R}^d$ and $\star :
\Lambda^n(\mathbb{R}^d) \rightarrow \Lambda^{d-n}(\mathbb{R}^d)$
will denote the Hodge star operator. (It is worth pointing out here
that if the reader is prepared to sacrifice the explicit
diffeomorphism-invariance that we seek, then they may effectively
dispense with these exterior algebraic considerations.) Given
$(\mathbf{B},\mathbf{p}) \in \mathcal{C}$ define
$X_j(B_j)\in\Lambda^{d_j}(\mathbb{R}^d)$ to be the wedge product of
the rows of the $d_j\times d$ matrix $B_j$. By \eqref{e:directsum}
it follows that
\begin{equation} \label{e:transversalityHodge}
\star\bigwedge_{j=1}^m\star X_j(B_j)\in\mathbb{R}\backslash\{0\}.
\end{equation}
The quantity in \eqref{e:transversalityHodge} is a certain
determinant and should be viewed as a means of quantifying the
transverality hypothesis \eqref{e:directsum}.
\begin{proposition}\label{p:linearcase}
If $(\mathbf{B},\mathbf{p}) \in \mathcal{C}$ then
$$
\emph{BL}(\mathbf{B},\mathbf{p}) = \left|\star\bigwedge_{j=1}^m\star
X_j(B_j)\right|^{-\frac{1}{m-1}},
$$
and thus
\begin{equation}\label{linadmis}
\int_{\mathbb{R}^d}\prod_{j=1}^m f_j(B_j x)^{\frac{1}{m-1}}\,\mathrm{d}x\leq
\left|\star\bigwedge_{j=1}^m\star
X_j(B_j)\right|^{-\frac{1}{m-1}}\prod_{j=1}^m\left(\int_{\mathbb{R}^{d_j}}f_j\right)^{\frac{1}{m-1}}
\end{equation}
for all nonnegative $f_j\in L^1(\mathbb{R}^{d_j})$, $1\leq j\leq m$.
\end{proposition}
One may reduce Proposition \ref{p:linearcase} to Proposition
\ref{p:Finnerorth} by appropriate linear changes of variables; see
Appendix \ref{appendix:reduction} for full details of this argument
which will be of further use in Section \ref{section:canonical} for
the nonlinear case.

Since the inequality \eqref{linadmis} is affine-invariant, one
should expect it to have a diffeomorphism-invariant nonlinear
version. This is our main result with regard to nonlinear
generalisations of Brascamp--Lieb inequalities.
\begin{theorem}\label{t:main} Let $\beta, \varepsilon, \kappa  >0$ be given. Suppose
that $B_j:\mathbb{R}^d\rightarrow\mathbb{R}^{d_j}$ is a
$C^{1,\beta}$ submersion satisfying
$\|B_j\|_{C^{1,\beta}}\leq\kappa$ in a neighbourhood of a point
$x_0\in\mathbb{R}^d$ for each $1\leq j\leq m$. Suppose further that
\begin{equation} \label{e:directsumnonlin}
\bigoplus_{j=1}^m \ker \mathrm{d}B_j(x_0)=\mathbb{R}^d
\end{equation}
and
$$
\left|\star\bigwedge_{j=1}^m\star X_j(\mathrm{d}B_j(x_0))\right| \geq \varepsilon.
$$
Then there exists a neighbourhood $U$ of $x_0$ depending on at most
$\beta, \varepsilon, \kappa$ and $d$, such that for all cutoff
functions $\psi$ supported in $U$, there is a constant $C$ depending
only on $d$ and $\psi$ such that
\begin{equation}\label{linadmisnl}
\int_{\mathbb{R}^d}\prod_{j=1}^m f_j(B_j (x))^{\frac{1}{m-1}}\psi(x)\,\mathrm{d}x\leq
C\varepsilon^{-\frac{1}{m-1}}\prod_{j=1}^m\left(\int_{\mathbb{R}^{d_j}}f_j\right)^{\frac{1}{m-1}}
\end{equation}
for all nonnegative $f_j\in L^1(\mathbb{R}^{d_j})$, $1\leq j\leq m$.
\end{theorem}
Inequality \eqref{linadmisnl} may be interpreted as a multilinear
``Radon-like" transform estimate. This is made explicit in the
following corollary, upon which our applications in Section
\ref{section:applications} depend.
\begin{corollary}\label{lemma6}
Let $\beta, \varepsilon, \kappa > 0$ be given. If
$F:(\mathbb{R}^{d-1})^{d-1}\rightarrow\mathbb{R}$ is such that
$\|F\|_{C^{1,\beta}} \leq \kappa$ and
$$
|\det(\nabla_{u_1}F(0),\ldots,\nabla_{u_{d-1}}F(0))| \geq
\varepsilon,$$ then there exists a neighbourhood $V$ of the origin
in $(\mathbb{R}^{d-1})^{d-1}$, depending only on $\beta,
\varepsilon,\kappa$ and $d$, and a constant $C$ depending only on
$d$, such that
\begin{equation}\label{lemma6est}
\int_{V}f_1(u_1)\cdots f_{d-1}(u_{d-1})f_d(u_1+\cdots+u_{d-1})\delta(F(u))\,\mathrm{d}u \leq C\varepsilon^{-\frac{1}{d-1}}\prod_{j=1}^d\|f_j\|_{(d-1)'}
\end{equation}
for all nonnegative $f_j \in L^{(d-1)'}(\mathbb{R}^{d-1})$, $1 \leq
j \leq m$.
\end{corollary}
The case $d=3$ of Corollary \ref{lemma6} was proved in \cite{BCW} as
a consequence of the nonlinear Loomis--Whitney inequality.

It is perhaps interesting to view Corollary \ref{lemma6} in the
light of the theory of multilinear weighted convolution inequalities
for $L^2$ functions developed in \cite{taoweighted}. Inequality
\eqref{lemma6est} is an example of such a convolution inequality in
an $L^p$ setting and with a singular (distributional) weight.

We conclude this section with a number of remarks on Theorem
\ref{t:main}.

As in the reduction of Proposition \ref{p:linearcase} to Proposition
\ref{p:Finnerorth}, a linear change of variables argument shows that
Theorem \ref{t:main} may be reduced to the case where each linear
mapping $\mathrm{d}B_j(x_0)$ is equal to the coordinate projection
$\Pi_j$ given by \eqref{e:Pijdefn}, in which case
$$
\star \bigwedge_{j=1}^m \star X_j(\mathrm{d}B_j(x_0)) = 1.
$$
Although this reduction is not essential, it does lead to some
conceptual and notational simplification in the subsequent analysis.
The details of this reduction may be found in Section
\ref{section:canonical}.

The core component of the proof of Theorem \ref{t:main} that we
present is based on \cite{BHT} and uses the idea of
induction-on-scales. This approach provides additional information
about the sizes of the neighbourhood $U$ and constant $C$ appearing
in its statement; see Section \ref{section:canonical} for further
details of this. In Section \ref{section:C0} we offer an explanation
of why the induction-on-scales approach is natural in the context of
Brascamp--Lieb inequalities and why the class $\mathcal{C}$ given in
\eqref{e:C0} is a natural class for consideration. In Section
\ref{section:proofidea}, we provide an outline of the proof of
Theorem \ref{t:main} which should guide the reader through the full
proof which is contained in Sections \ref{section:canonical} and
\ref{section:induction}.

In the case where $d_j=d-1$ for all $j$, Theorem \ref{t:main}
reduces to the nonlinear Loomis--Whitney inequality in \cite{BCW}
except that the stronger hypothesis $B_j\in C^3$ is assumed in
\cite{BCW}. The proof of the result in \cite{BCW} is quite different
from the proof we give here, and is based on the so-called method of
refinements of M. Christ \cite{Ch}.
We make some further remarks on the role of the smoothness of the
mappings $B_j$ at the end of Section \ref{section:induction}.

The condition \eqref{e:directsumnonlin} is somewhat less restrictive
than it may appear. For example, consider smooth mappings
$B_j:\mathbb{R}^5\rightarrow\mathbb{R}^2$ satisfying
$$
\ker \mathrm{d}B_j(x_0)=\langle\{e_j,e_{(j+1)\text{mod}\, 5}, e_{(j+2)\text{mod}\, 5}\}\rangle
$$
for each $1\leq j\leq 5$, where $e_j$ denotes the $j$th standard
basis vector in $\mathbb{R}^5$. Evidently the condition
\eqref{e:directsumnonlin} is not satisfied. However we may write
$$
\prod_{j=1}^5(f_j\circ B_j)^{1/2}=\prod_{j=1}^5(\widetilde{f}_j\circ \widetilde{B}_j)^{1/4},
$$
where $\widetilde{f}_j:=f_j\otimes f_{(j+2) \text{mod}\,
5}:\mathbb{R}^4\rightarrow [0,\infty)$ and
$\widetilde{B}_j:=(B_j,B_{(j+2)\text{mod}\,
5}):\mathbb{R}^5\rightarrow\mathbb{R}^4$. Since $\ker
\mathrm{d}\widetilde{B}_j(x_0)=\langle\{e_{(j+2) \text{mod}\,
5}\}\rangle$ for each $1\leq j\leq 5$, the mappings
$\widetilde{B}_j$ do satisfy the condition
\eqref{e:directsumnonlin}, and so by Theorem \ref{t:main}
\begin{eqnarray*}
\begin{aligned}
\int_{\mathbb{R}^5}\prod_{j=1}^5(f_j\circ B_j)^{1/2}\psi&=\int_{\mathbb{R}^5}\prod_{j=1}^5(\widetilde{f}_j\circ \widetilde{B}_j)^{1/4}\psi\\
& \leq C\prod_{j=1}^5\left(\int_{\mathbb{R}^4}\widetilde{f}_j\right)^{1/4}\\&=C\prod_{j=1}^5\left(\int_{\mathbb{R}^2}f_j\right)^{1/2}.
\end{aligned}
\end{eqnarray*}
Here the cutoff function $\psi$ and constant $C$ are as in the
statement of Theorem \ref{t:main}. This inequality is optimal in the
sense that $\mbox{BL}(\mathrm{d}B_j(x_0)),(p_j))<\infty$ if and only
if $p_1=\cdots=p_5=1/2$ -- see \cite{Finner}. Similar considerations
form an important part of the proof of Corollary \ref{lemma6} in
dimensions $d\geq 4$.

Very recently, Stovall \cite{Stovall} considered inequalities of the
type \eqref{aspiration} for the case $d_j = d-1$ for all $j$ where
one does not necessarily have the transversality hypothesis
\eqref{e:directsumnonlin}. Here, curvature of the fibres of the
$B_j$ plays a crucial role. In \cite{Stovall}, Stovall determined
completely all data $(\mathbf{B},\mathbf{p})$, up to endpoints in
$\mathbf{p}$, for which inequality \eqref{aspiration} holds when
each $B_j : \mathbb{R}^d \rightarrow \mathbb{R}^{d-1}$ is a smooth
submersion.  The work in \cite{Stovall} generalised work of Tao and
Wright \cite{TW} for the bilinear case $m=2$, and both approaches are
based on Christ's method of refinements. It would be interesting to
complete the picture further and understand the case where one does
not necessarily have transversality and each $d_j$ is not
necessarily equal to $d-1$. We do not pursue this matter here.

Given that Theorem \ref{t:main} is a local result it is natural to
ask whether one may obtain global versions based on the assumption
that hypothesis \eqref{e:directsum} holds at every point
$x_0\in\mathbb{R}^d$, possibly with the insertion of a suitable
weight factor. Simple examples show that naive versions, involving
weights which are powers of the quantity
$\star\bigwedge_{j=1}^m\star X_j(\mathrm{d}B_j(x))$ cannot hold; see
\cite{BCW} for an explicit example.

\subsection*{Organisation of the paper} To recap, in the next section
we give some justification for our choice of proof of Theorem
\ref{t:main} and the class $\mathcal{C}$. In Section
\ref{section:proofidea} we give an outline of the proof of Theorem
\ref{t:main} by considering the special case of the nonlinear
Loomis--Whitney inequality in three dimensions. The full proof
begins in Section \ref{section:canonical} where we make the
reduction to the coordinate projection case. The proof for this case
rests on the induction-on-scales argument which appears in Section
\ref{section:induction}. 
In Section \ref{section:corproof} we give a proof of Corollary
\ref{lemma6}, and in Section \ref{section:applications} we provide
applications to two closely related problems in harmonic analysis.

\subsection*{Acknowledgements} The authors would like to express
gratitude to the anonymous referee for their careful reading of the
manuscript and extremely helpful recommendations, and also to Steve
Roper at the University of Glasgow for creating the figures in
Section \ref{section:proofidea}.

\section{Induction-on-scales and the class $\mathcal{C}$} \label{section:C0}
The Brascamp--Lieb inequalities \eqref{bl} possess a certain
self-similar structure that strongly suggests an approach to the
corresponding nonlinear statements by induction-on-scales.
Induction-on-scales arguments have been used with great success in
harmonic analysis in recent years. Very closely related to the
forthcoming discussion is the induction-on-scales approach to the
Fourier restriction and Kakeya conjectures originating in work of
Bourgain \cite{Bo}, and developed further by Wolff \cite{wolffcone}
and Tao \cite{taoparaboloid}; see also the survey article \cite{T}.
This self-similarity manifests itself most elegantly in an
elementary convolution inequality due to Ball \cite{Ball} (see also
\cite{BCCT}), which we now describe.

Let $(\mathbf{B},\mathbf{p})$ be a Brascamp--Lieb datum where each
$B_j$ is linear. Let $\mathbf{f}$ and $\mathbf{f}'$ be two inputs
and we assume, for clarity of exposition, that these inputs are
$L^1$-normalised. For each $x \in \mathbb{R}^d$ and $1\leq j\leq m$
let $g_j^x : \mathbb{R}^{d_j} \rightarrow [0,\infty)$ be given by
$$
g_j^x(y) = f_j(B_jx-y)f'_j(y).
$$
By Fubini's theorem and elementary considerations we have that
\begin{eqnarray*}
\begin{aligned}
\mbox{BL}(\textbf{B},\textbf{p};\textbf{f})\mbox{BL}(\textbf{B},\textbf{p};\textbf{f}')&=
\int_{\mathbb{R}^d}\prod_{j=1}^m(f_j\circ
B_j)^{p_j}*\prod_{j=1}^m(f'_j\circ B_j)^{p_j}\\
&=\int_{\mathbb{R}^d}\left(\int_{\mathbb{R}^d}\prod_{j=1}^m(g_j^x\circ
B_j)^{p_j}\right)\mathrm{d}x\\
&\leq \int_{\mathbb{R}^d}\left(\mbox{BL}(\textbf{B},\textbf{p};(g_j^x))\prod_{j=1}^m\left(\int_{\mathbb{R}^{d_j}}
g_j^x(y)\,\mathrm{d}y\right)^{p_j}\right)\mathrm{d}x \\
&= \int_{\mathbb{R}^d}\left(\mbox{BL}(\textbf{B},\textbf{p};(g_j^x))\prod_{j=1}^m(f_j*f_j'(B_jx))^{p_j}\right)\mathrm{d}x\\
\end{aligned}
\end{eqnarray*}
and therefore
\begin{equation} \label{e:Ballinequality}
\mbox{BL}(\textbf{B},\textbf{p};\textbf{f})\mbox{BL}(\textbf{B},\textbf{p};\textbf{f}')
\leq \sup_{x\in\mathbb{R}^d}\mbox{BL}(\textbf{B},\textbf{p};(g_j^x))\;\mbox{BL}(\textbf{B},\textbf{p};\textbf{f}*\textbf{f}'),
\end{equation}
where $\mathbf{f}*\mathbf{f}':=(f_j*f_j')$. Notice that if
$\mathbf{f}'$ is an extremiser to \eqref{bl}, i.e.
$$\mbox{BL}(\textbf{B},\textbf{p};\textbf{f}')=\mbox{BL}(\textbf{B},\textbf{p}),$$
then since
$$\mbox{BL}(\textbf{B},\textbf{p};\textbf{f}*\textbf{f}')\leq
\mbox{BL}(\textbf{B},\textbf{p}),$$ we may deduce that
\begin{equation}\label{ball1}
\mbox{BL}(\textbf{B},\textbf{p};\textbf{f})\leq
\sup_{x\in\mathbb{R}^d}\mbox{BL}(\textbf{B},\textbf{p};(g_j^x)).
\end{equation}
In particular, in the presence of an appropriately ``localising"
extremiser $\mathbf{f}'$ (such as of compact support), \eqref{ball1}
suggests the viability of a proof of nonlinear inequalities such as
\eqref{aspiration} by induction on the ``scale of the support" of
$\textbf{f}$. The point is that $g_j^x$ may be thought of as the
function $f_j$ localised by $f_j'$ to a neighbourhood of the general
point $B_jx$.

With the above discussion in mind it is natural to restrict
attention to data $(\mathbf{B},\mathbf{p})$ for which \eqref{bl} has
extremisers of the form $\mathbf{f}=(\chi_{E_j})$, where for each
$j$, $E_j$ is a subset of $\mathbb{R}^{d_j}$ which tiles by
translation. Furthermore, given our aspirations, it is natural to
choose a class of data which is affine-invariant and stable under
linear perturbations of $\mathbf{B}$. These requirements lead us to
the transversality hypothesis in \eqref{e:directsum}. Indeed, as
there are linear changes of variables which show that Proposition
\ref{p:linearcase} follows from Proposition \ref{p:Finnerorth} (see
Appendix \ref{appendix:reduction}), it is straightforward to observe
that characteristic functions of certain paralellepipeds are
extremisers for \eqref{linadmis}. Such sets of course tile by
translation.

We remark that there are other hypotheses on the datum $\mathbf{B}$
which fulfill our requirements. For example, one may replace
\eqref{e:directsum} by
\begin{equation*}
\bigoplus_{j=1}^m \coker B_j=\mathbb{R}^d.
\end{equation*}
However, after appropriate changes of variables, the corresponding
nonlinear inequality \eqref{aspiration} merely reduces to a
statement of Fubini's theorem, and in particular, $p_j=1$ for all
$j$. There are further alternatives which are hybrids of these and
are similarly degenerate.
\begin{remark}
Notice that if $\mathbf{f}'$ is an extremiser to \eqref{bl} then we
may also deduce from \eqref{e:Ballinequality} that
\begin{equation}\label{ball2}
\mbox{BL}(\textbf{B},\textbf{p};\textbf{f})\leq
\mbox{BL}(\textbf{B},\textbf{p};\textbf{f}*\textbf{f}').
\end{equation}
This inequality suggests the viability of a proof of nonlinear
inequalities such as \eqref{aspiration} by induction on the ``scale
of constancy" of $f$. Certain weak versions of inequality
\eqref{aspiration}, where the resulting constant $C$ has a mild
dependence on the smoothness of the input $\mathbf{f}$, have already
been treated in this way in \cite{BCT} (see Remarks 6.3 and 6.6).

In certain situations, \eqref{ball2} leads to the monotonicity of
$\mbox{BL}(\textbf{B},\textbf{p};\textbf{f})$ under the action of
convolution semigroups on the input $\textbf{f}$. In the context of
heat-flow, this observation originates in \cite{CLL} and
\cite{BCCT}; see the latter for further discussion of this
perspective.
\end{remark}

\section{An outline of the proof of Theorem \ref{t:main}}
\label{section:proofidea}

The purpose of this section is to bring out the key ideas in the
proof of Theorem \ref{t:main}. It is also an opportunity to
introduce some notation which will be adopted (modulo small
modifications) in the full proof in Section \ref{section:induction}.
As it is an outline we will sometimes compromise rigour for the sake
of clarity. Our approach is based on \cite{BHT}.

Since the induction-on-scales argument we use to prove Theorem
\ref{t:main} is guided by the underlying geometry, in this outline
we will consider the Loomis--Whitney case where $d=3$, $m=3$ and
\begin{equation} \label{e:LWnormalisation}
\mathrm{d}B_j(x_0) = \Pi_j
\end{equation}
for $j=1,2,3$. In particular, we have $\ker \mathrm{d}B_j(x_0) =
\langle e_j \rangle$ where $e_j$ denotes the $j$th standard basis
vector in $\mathbb{R}^3$.

We shall use $Q(x,\delta)$ to denote the axis-parallel cube centred
at $x$ with sidelength equal to $\delta$.

Fix a small sidelength $\delta_0>0$ which, in terms of the
induction-on-scales argument, represents the largest or ``global"
scale.

For $\delta, M
> 0$ we let $C(\delta,M)$ denote the best constant in the inequality
$$
\int_Q f_1(B_1(x))^{\frac{1}{2}}f_2(B_2(x))^{\frac{1}{2}}f_3(B_3(x))^{\frac{1}{2}}\,\mathrm{d}x \leq C \bigg( \int_{\mathbb{R}^2} f_1 \bigg)^{\frac{1}{2}}
\bigg( \int_{\mathbb{R}^2} f_2 \bigg)^{\frac{1}{2}}
\bigg( \int_{\mathbb{R}^2}  f_3 \bigg)^{\frac{1}{2}}
$$
over all axis-parallel subcubes $Q$ of $Q(x_0,\delta_0)$ of
sidelength $\delta$ and all inputs $f_1, f_2, f_3 \in
L^1(\mathbb{R}^2)$ which are ``constant" at the scale $M^{-1}$. The
goal is to prove that $C(\delta_0,M)$ is bounded above by a constant
independent of $M$, allowing the use of a density argument to pass
to general $f_1, f_2, f_3 \in L^1(\mathbb{R}^2)$.

As our proof proceeds by induction it consists of two distinct
parts.
\begin{enumerate}
\item[(i)] \emph{The base case}: For each $M>0$, $C(\delta,M)$ is bounded by an absolute constant for all $\delta$ sufficiently small.
\item[(ii)] \emph{The inductive step}: There exists $\gamma>0$ and $\alpha>1$ such that \begin{equation}\label{inductionoutline}C(\delta,M)\leq (1+O(\delta^\gamma))C(2\delta^\alpha,M)\end{equation} uniformly in $\delta\leq\delta_0$ and $M>0$.
\end{enumerate}

Claims (i) and (ii) quickly lead to the desired conclusion since on
iterating \eqref{inductionoutline} we find that $C(\delta_0,M)$ is
bounded by a convergent product of factors of the form
$(1+O(\delta^\gamma))$ with $\delta\leq\delta_0$.

To see why the base case is true, let $Q$ be any axis-parallel cube
contained in $Q(x_0,\delta_0)$ with centre $x_Q$ and sidelength
$\delta$, and let $f_1, f_2, f_3 \in L^1(\mathbb{R}^2)$ be constant
at scale $M^{-1}$. Observe that if $\delta$ is sufficiently small
then each $f_j$ does not ``see" the difference between $B_j(x)$ and
$\mathrm{d}B_j(x_Q)x$ for $x \in Q$ in the sense that $f_j\circ
B_j\sim f_j\circ \mathrm{d}B_j(x_Q)$ (up to harmless translations)
on $Q$. Now, by \eqref{e:LWnormalisation} and the smoothness of the
$B_j$ we know that
$$|X_j(\mathrm{d}B_j(x_Q)) -
e_j| = |X_j(\mathrm{d}B_j(x_Q)) - X_j(\Pi_j)| \leq 1/10
$$
if $\delta_0$ is sufficiently small. Hence by Proposition
\ref{p:linearcase} it follows that $C(\delta,M)$ is bounded above by
an absolute constant for such $\delta$.

Turning to the inductive step, fix any axis-parallel cube $Q$
contained in $Q(x_0,\delta_0)$ with centre $x_Q$ and sidelength
$\delta$, and let $f_1, f_2, f_3 \in L^1(\mathbb{R}^2)$ be constant
at scale $M^{-1}$. First we decompose $Q = \bigcup P(n)$, where the
$P(n)$ are axis-parallel subcubes with equal sidelength
$\delta^\alpha$, and $\alpha>1$. We choose the natural indexing of
the $P(n)$ by $n \in \mathbb{N}^3$. Unfortunately this decomposition
is too naive to prove the inductive step but nevertheless it is
instructive to see where the proof breaks down.

Observe that
\begin{align} \label{e:hmm}
& \int_Q f_1(B_1(x))^{\frac{1}{2}}f_2(B_2(x))^{\frac{1}{2}}f_3(B_3(x))^{\frac{1}{2}}\,\mathrm{d}x \notag \\
 = & \sum_{n \in \mathbb{N}^3} \int_{P(n)} f_1(B_1(x))^{\frac{1}{2}}f_2(B_2(x))^{\frac{1}{2}}f_3(B_3(x))^{\frac{1}{2}}\,\mathrm{d}x \notag \\
 \leq & C(\delta^\alpha,M) \sum_{n \in \mathbb{N}^3} \bigg( \int_{B_1(P(n))} f_1 \bigg)^{\frac{1}{2}} \bigg( \int_{B_2(P(n))} f_2 \bigg)^{\frac{1}{2}}
\bigg( \int_{B_3(P(n))} f_3 \bigg)^{\frac{1}{2}}.
\end{align}
If $n = (n_1,n_2,n_3)$ then $\int_{B_1(P(n))} f_1$ is ``almost" a
function of $n_2$ and $n_3$. Indeed, if $B_1$ is linear and equal to
$\Pi_1$ then
$$
B_1(P(n)) = B_1(T_1(n_2,n_3))
$$
where $T_1(n_2,n_3)$ is a cuboid (or ``tube") with long side in the
direction of $e_1$ and containing $P(n)$. A similar remark holds for
$\int_{B_2(P(n))} f_2$ and $\int_{B_3(P(n))} f_3$.

For $j=1,2,3$ this leads us to define cuboids
$$
T_j(\ell) = \bigcup_{\substack{n \in \mathbb{N}^3:\\ \Pi_jn = \ell}} P(n)
$$
for $\ell \in \mathbb{N}^2$. Note that $T_j(\ell)$ has direction
$e_j$ and its location is determined by $\ell \in \mathbb{N}^2$. In
particular, for each $n \in \mathbb{N}^3$, $T_j(\Pi_jn)$ is a cuboid
in the direction $e_j$ which passes through $P(n)$. See Figure
\ref{fig1}.
\begin{figure}
\centering
\psset{unit=1.35cm} \psset{THETA=-45,PHI=30,Dobs=15,Decran=15}
\begin{pspicture}(-1,-1)(5,4)

  \pNodeThreeD(1,0,0){A}
  \pNodeThreeD(2,0,0){B}
  \pNodeThreeD(2,1,0){C}
  \pNodeThreeD(1,1,0){D}

  \pNodeThreeD(1,0,1){E}
  \pNodeThreeD(2,0,1){F}
  \pNodeThreeD(2,1,1){G}
  \pNodeThreeD(1,1,1){H}

  \pNodeThreeD(1,0,2){I}
  \pNodeThreeD(2,0,2){J}
  \pNodeThreeD(2,1,2){K}
  \pNodeThreeD(1,1,2){L}

  \pNodeThreeD(1,0,3){M}
  \pNodeThreeD(2,0,3){N}
  \pNodeThreeD(2,1,3){O}
  \pNodeThreeD(1,1,3){P}

  \newrgbcolor{lg}{0.7 0.7 0.7}

  \pspolygon[fillstyle=solid,fillcolor=lg](A)(E)(F)(B)
  \pspolygon[fillstyle=solid,fillcolor=lg](B)(F)(G)(C)
  \pspolygon[fillstyle=solid,fillcolor=lg](E)(I)(J)(F)
  \pspolygon[fillstyle=solid,fillcolor=lg](F)(J)(K)(G)
  \pspolygon[fillstyle=solid,fillcolor=lg](I)(M)(N)(J)
  \pspolygon[fillstyle=solid,fillcolor=lg](J)(N)(O)(K)
  \pspolygon[fillstyle=solid,fillcolor=lg](M)(N)(O)(P)

  \pNodeThreeD(2,2,2){A}
  \pNodeThreeD(3,2,2){B}
  \pNodeThreeD(3,3,2){C}
  \pNodeThreeD(2,3,2){D}

  \pNodeThreeD(2,2,3){E}
  \pNodeThreeD(3,2,3){F}
  \pNodeThreeD(3,3,3){G}
  \pNodeThreeD(2,3,3){H}

  \pspolygon[fillstyle=solid,fillcolor=lg](A)(E)(F)(B)
  \pspolygon[fillstyle=solid,fillcolor=lg](B)(F)(G)(C)
  \pspolygon[fillstyle=solid,fillcolor=lg](E)(F)(G)(H)

\multido{\iX=0+1}{4}{%
    \pNodeThreeD(\iX,0,0){S}
    \pNodeThreeD(\iX,0,3){E}
    \psline(S)(E)
}

\multido{\iZ=0+1}{4}{%
    \pNodeThreeD(0,0,\iZ){S}
    \pNodeThreeD(3,0,\iZ){E}
    \psline(S)(E)
}

\multido{\iY=0+1}{4}{%
    \pNodeThreeD(3,\iY,0){S}
    \pNodeThreeD(3,\iY,3){E}
    \psline(S)(E)
}

\multido{\iZ=0+1}{4}{%
    \pNodeThreeD(3,0,\iZ){S}
    \pNodeThreeD(3,3,\iZ){E}
    \psline(S)(E)
}

\multido{\iX=0+1}{4}{%
    \pNodeThreeD(\iX,0,3){S}
    \pNodeThreeD(\iX,3,3){E}
    \psline(S)(E)
}

\multido{\iY=0+1}{4}{%
    \pNodeThreeD(0,\iY,3){S}
    \pNodeThreeD(3,\iY,3){E}
    \psline(S)(E)
}

\pNodeThreeD(0,0,3.5){S} \pNodeThreeD(0,0,4.0){E}
\psline[arrows=->](S)(E) \pNodeThreeD(0,0,3.75){L}
\uput[r](L){$e_3$}

\pNodeThreeD(1.5,0,0.5){E} \pNodeThreeD(1.5,-0.5,0){S}
\psline[arrows=->](S)(E) \uput[l](S){$T_3(\ell)$}

\pNodeThreeD(3,2.5,2.5){E} \pNodeThreeD(3.5,3,2.5){S}
\psline[arrows=->](S)(E) \uput[r](S){$P(n)$}

\end{pspicture}
\vspace{5mm} \caption{Subcubes $P(n)$ parametrised by $n \in
\mathbb{N}^3$ and tubes $T_3(\ell)$ parametrised by $\ell \in
\mathbb{N}^2$ with direction $e_3$.} \label{fig1}
\end{figure}
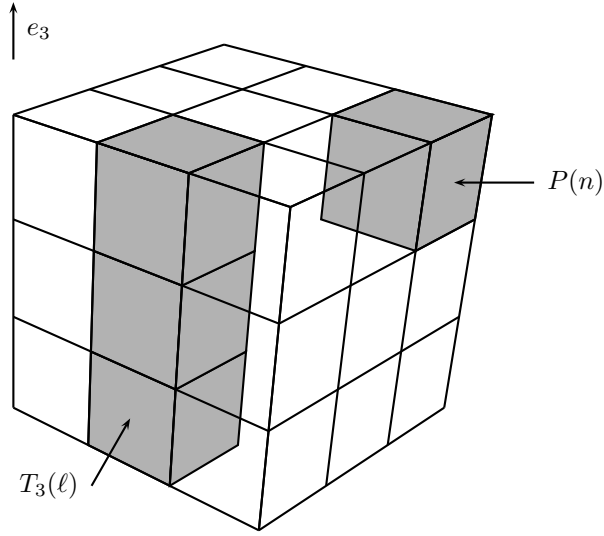
Accordingly, we define
$$
F_j(\ell) = \int_{B_j(T_j(\ell))} f_j
$$
for $j=1,2,3$ and $\ell \in \mathbb{N}^2$. Then by \eqref{e:hmm} and
the discrete inequality \eqref{e:discreteFinner},
\begin{eqnarray*}
\begin{aligned}
\int_Q f_1(B_1(x))^{\frac{1}{2}}f_2(B_2(x))^{\frac{1}{2}}f_3(B_3(x))^{\frac{1}{2}}\,\mathrm{d}x \leq &
C(\delta^\alpha,M) \sum_{n \in \mathbb{N}^3} F_1(\Pi_1n)^{\frac{1}{2}} F_2(\Pi_2n)^{\frac{1}{2}}
F_3(\Pi_3n)^{\frac{1}{2}} \\
\leq & C(\delta^\alpha,M) \|F_1\|^{\frac{1}{2}}_{\ell^1(\mathbb{N}^2)}
\|F_2\|_{\ell^1(\mathbb{N}^2)}^{\frac{1}{2}} \|F_3\|_{\ell^1(\mathbb{N}^2)}^{\frac{1}{2}}.
\end{aligned}
\end{eqnarray*}
If we had disjointness in the sense that
\begin{equation} \label{e:LWdisjoint}
B_j(T_j(\ell)) \cap B_j(T_j(\ell')) = \emptyset \qquad \text{whenever}\qquad \ell \neq \ell',
\end{equation}
then
\begin{equation*}
\|F_j\|_{\ell^1(\mathbb{N}^2)} \leq  \int_{\mathbb{R}^2} f_j
\end{equation*}
would hold for each $j=1,2,3$, and hence
\begin{equation} \label{e:perfect}
\int_Q f_1(B_1(x))^{\frac{1}{2}}f_2(B_2(x))^{\frac{1}{2}}f_3(B_3(x))^{\frac{1}{2}}\,\mathrm{d}x \leq C(\delta^\alpha,M)
\bigg( \int_{\mathbb{R}^2} f_1 \bigg)^{\frac{1}{2}}
\bigg(  \int_{\mathbb{R}^2} f_2 \bigg)^{\frac{1}{2}}
\bigg(  \int_{\mathbb{R}^2} f_3 \bigg)^{\frac{1}{2}}
\end{equation}
would follow immediately. If each $B_j$ is linear and equal to
$\Pi_j$ then \eqref{e:LWdisjoint} is of course true, although
otherwise it is not. In order to achieve a version of
\eqref{e:LWdisjoint} in general, it is necessary to modify our
decomposition of $Q$.

To better understand the location of each image $B_j(T_j(\Pi_jn))$
the $P(n)$ should in fact be parallelepipeds whose faces are given
by pull-backs of certain lines in $\mathbb{R}^2$ under the linear
maps $\mathrm{d}B_j(x_Q)$.

However, we still need to fully accommodate for the nonlinearity and
in particular the difference between $B_j(T_j(\ell))$ and
$\mathrm{d}B_j(x_Q)(T_j(\ell))$. Following the approach in
\cite{BHT} it is natural to insert relatively narrow ``buffer zones"
between the $P(n)$ to provide sufficient separation in order to
guarantee the sought after disjointness property
\eqref{e:LWdisjoint}. Clearly this depends on the smoothness of the
$B_j$ and, since we assume $C^{1,\beta}$ regularity, we take the
$P(n)$ to have sidelengths approximately $\delta^{\alpha_0}$ and the
buffer zones to have width approximately $\delta^{\alpha_1}$ where
$$
1 < \alpha_0 < \alpha_1 < 1 + \beta.
$$

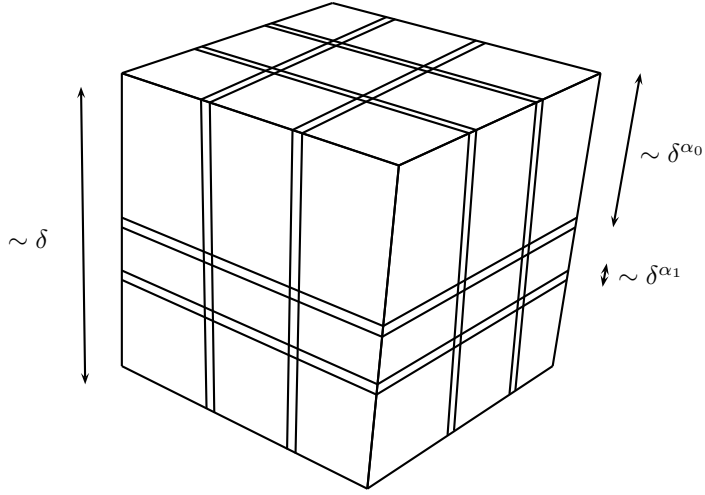
\begin{figure}
\psset{unit=1.35cm} \psset{THETA=-45,PHI=30,Dobs=15,Decran=15}
\begin{center}
\begin{pspicture}(-1,-1)(5,4)
  \multido{\iX=0+3}{2}{%
    \pNodeThreeD(\iX,0,0){S}
    \pNodeThreeD(\iX,0,3){E}
    \psline(S)(E)
    \pNodeThreeD(0,0,\iX){S}
    \pNodeThreeD(3,0,\iX){E}
    \psline(S)(E)
    \pNodeThreeD(3,\iX,0){S}
    \pNodeThreeD(3,\iX,3){E}
    \psline(S)(E)
    \pNodeThreeD(3,0,\iX){S}
    \pNodeThreeD(3,3,\iX){E}
    \psline(S)(E)
    \pNodeThreeD(\iX,0,3){S}
    \pNodeThreeD(\iX,3,3){E}
    \psline(S)(E)
    \pNodeThreeD(0,\iX,3){S}
    \pNodeThreeD(3,\iX,3){E}
    \psline(S)(E)
  }

  \multido{\nX=0.95+0.10}{2}{%
    \pNodeThreeD(\nX,0.2,-0.17){S}
    \pNodeThreeD(\nX,0,3){E}
    \psline(S)(E)
    \pNodeThreeD(0,0,\nX){S}
    \pNodeThreeD(3,0,\nX){E}
    \psline(S)(E)
    \pNodeThreeD(3.15,\nX,0.19){S}
    \pNodeThreeD(3,\nX,3){E}
    \psline(S)(E)
    \pNodeThreeD(3,0,\nX){S}
    \pNodeThreeD(3,3,\nX){E}
    \psline(S)(E)
    \pNodeThreeD(\nX,0,3){S}
    \pNodeThreeD(\nX,2.6,3.15){E}
    \psline(S)(E)
    \pNodeThreeD(0,\nX,3){S}
    \pNodeThreeD(3,\nX,3){E}
    \psline(S)(E)
}

 \multido{\nX=1.95+0.10}{2}{%
    \pNodeThreeD(\nX,0.2,-0.17){S}
    \pNodeThreeD(\nX,0,3){E}
    \psline(S)(E)
    \pNodeThreeD(3.15,\nX,0.19){S}
    \pNodeThreeD(3,\nX,3){E}
    \psline(S)(E)
    \pNodeThreeD(\nX,0,3){S}
    \pNodeThreeD(\nX,2.6,3.15){E}
    \psline(S)(E)
    \pNodeThreeD(0,\nX,3){S}
    \pNodeThreeD(3,\nX,3){E}
    \psline(S)(E)
}

 \multido{\nX=1.50+0.10}{2}{%
    \pNodeThreeD(0,0,\nX){S}
    \pNodeThreeD(3,0,\nX){E}
    \psline(S)(E)
    \pNodeThreeD(3,0,\nX){S}
    \pNodeThreeD(3,3,\nX){E}
    \psline(S)(E)
}

\pNodeThreeD(3.25,3.25,1.5){S} \pNodeThreeD(3.25,3.25,3.0){E}
\psline[arrows=<->](S)(E)

\pNodeThreeD(3.25,3.25,2.25){L} \uput[r](L){$\sim\delta^{\alpha_0}$}

\pNodeThreeD(3.25,3.25,0.875){S} \pNodeThreeD(3.25,3.25,1.125){E}
\psline[arrows=<->](S)(E)

\pNodeThreeD(3.25,3.25,1.0){L} \uput[r](L){$\sim\delta^{\alpha_1}$}

\pNodeThreeD(0,-0.5,0){S} \pNodeThreeD(0,-0.5,3.0){E}
\psline[arrows=<->](S)(E) \pNodeThreeD(0,-1.55,1.9){L}
\uput[r](L){$\sim\delta$}

\end{pspicture}
\end{center}
\vspace{5mm} \caption{The modified decomposition of $Q$.}
\label{fig2}
\end{figure}
The decomposition of $Q$ now has a ``main component" from the $P(n)$
and a ``error component" from the buffer zones. We would like to use
the above argument which led to \eqref{e:perfect} on each component.
However, in order for the error component to genuinely contribute an
acceptable error term, we need to relax the regular decomposition
(into equally sized $P(n)$) since a ``large" amount of mass of the
$f_j\circ B_j$ may lie on the buffer zones. Again following ideas
from \cite{BHT} we use a simple pigeonholing argument to position
the buffer zones in an efficient location given the constraint that
the $P(n)$ should have \emph{essentially} the same sidelengths. See
Figure \ref{fig2}. Putting the resulting estimates together yields
the desired recursive inequality \eqref{inductionoutline} with
$\alpha=\alpha_0$ and some $\gamma>0$.

See Section \ref{section:induction} for the complete details of this
induction-on-scales argument in the full generality of Theorem
\ref{t:main}.

\section{Preparation and reduction to the orthogonal projection case} \label{section:canonical}
Recall the definition of $\Pi_j : \mathbb{R}^d \rightarrow
\mathbb{R}^{d_j}$ given by \eqref{e:Pijdefn}. In this section we
shall prove that Theorem \ref{t:main} is a consequence of the
following nonlinear version of Proposition \ref{p:Finnerorth}.
\begin{proposition} \label{p:straightBj}
Suppose $\beta, \kappa >0$ are given and $\alpha_0,\alpha_1$ satisfy
$1 < \alpha_0 < \alpha_1 < 1+\beta$. Let
\begin{equation} \label{e:R0defn}
  \delta_0 = \min \left\{\bigg( \frac{c_d}{\kappa} \bigg)^{\frac{1}{1+\beta-\alpha_1}},
\bigg( \frac{1}{4}\bigg)^{\frac{1}{\min\{\alpha_0-1,\alpha_1-\alpha_0\}}} \right\}.
\end{equation}
Suppose that $B_j:\mathbb{R}^d\rightarrow\mathbb{R}^{d_j}$ is a
$C^{1,\beta}$ submersion satisfying
$\|B_j\|_{C^{1,\beta}}\leq\kappa$ in $Q(x_0,\delta_0)$ and
$\mathrm{d}B_j(x_0) = \Pi_j$ for each $1 \leq j \leq m$. Then for
$c_d \in (0,\kappa)$ sufficiently small,
\begin{equation*}
\int_{Q(x_0,\delta_0)} \prod_{j=1}^m f_j(B_j(x))^{\frac{1}{m-1}}\,\mathrm{d}x \leq
10^d \exp\left(\frac{10^d\delta_0^{\frac{\alpha_1-\alpha_0}{m-1}}}{1-2^{-\frac{\alpha_1-\alpha_0}{m-1}}}\right) \prod_{j=1}^m\left(\int_{\mathbb{R}^{d_j}}f_j\right)^{\frac{1}{m-1}}
\end{equation*}
for all nonnegative $f_j \in L^1(\mathbb{R}^{d_j})$, $1\leq j\leq
m$.
\end{proposition}

As mentioned already in the previous section, the proof of
Proposition \ref{p:straightBj} will proceed by an
induction-on-scales argument. For a cube at scale $\delta$, we
decompose into parallelepipeds of sidelength approximately
$\delta^{\alpha_0}$ and the buffer zones will have thickness
approximately $\delta^{\alpha_1}$. We have stated Proposition
\ref{p:straightBj} with this in mind and we have provided explicit
information on how the size of the neighbourhood and the constant
depend on the relevant parameters.

\subsection*{Deduction of Theorem \ref{t:main} from Proposition \ref{p:straightBj}}

The argument which follows is similar to the argument given in
Appendix \ref{appendix:reduction} for the corresponding claim in the
linear case. A little extra work is required to verify the
uniformity claims in Theorem \ref{t:main} concerning the
neighbourhood and the constant.

Select any set of vectors $\{a_k : k \in \mathcal{K}_j\}$ forming an
orthonormal basis for $\ker \mathrm{d}B_j(x_0)$. By definition of
the Hodge star and orthogonality we get
\begin{equation} \label{e:mainstar}
 \star X_j(\mathrm{d}B_j(x_0)) = \|X_j(\mathrm{d}B_j(x_0))\|_{\Lambda^{d_j}(\mathbb{R}^d)} \bigwedge_{k \in \mathcal{K}_j} a_k.
\end{equation}
Let $A$ be the $d \times d$ matrix whose $i$th column is equal to
$a_i$ for each $1 \leq i \leq d$.  Finally, let $C_j$ be the $d_j
\times d_j$ matrix given by
\begin{equation*}
C_j = \mathrm{d}B_j(x_0)A_j,
\end{equation*}
where $A_j$ is the $d \times d_j$ matrix obtained by deleting from
$A$ the columns $a_k$ for each $k \in \mathcal{K}_j$.

Then, by construction, the map $\widetilde{B}_j : \mathbb{R}^d
\rightarrow \mathbb{R}^{d_j}$ given by
\begin{equation*} \label{e:Bjfirst}
\widetilde{B}_j(x) = C_j^{-1}B_j(Ax)
\end{equation*}
satisfies
\begin{equation} \label{e:canonical}
\mathrm{d}\widetilde{B}_j(\widetilde{x}_0) = C_j^{-1}\mathrm{d}B_j(x_0)A = \Pi_j,
\end{equation}
where $\widetilde{x}_0 = A^{-1}x_0$. Since we are assuming
\eqref{e:directsum} and since $B_j$ is a submersion at $x_0$ we know
that the matrices $A$ and $C_j$ are invertible.

Let $U$ be some neighbourhood of $x_0$ and $\psi$ a cutoff function
supported in $U$. Using $A$ to change variables one obtains
\begin{equation} \label{e:A}
\int_{\mathbb{R}^d} \prod_{j=1}^m f_j(B_j(x))^{\frac{1}{m-1}}\psi(x) \,\mathrm{d}x = |\det(A)| \int_{\mathbb{R}^d} \prod_{j=1}^m \widetilde{f}_j(\widetilde{B}_j(x))^{\frac{1}{m-1}}\widetilde{\psi}(x) \,\mathrm{d}x,
\end{equation}
where $\widetilde{\psi} = \psi \circ A$ is a cutoff function
supported in $A^{-1}U$ and $\widetilde{f}_j = f_j \circ C_j$, $1
\leq j \leq m$. Of course, we know that
$\mathrm{d}\widetilde{B}_j(\widetilde{x}_0) = \Pi_j$ by
\eqref{e:canonical}. Notice also that
\begin{align*}
\|\mathrm{d}\widetilde{B}_j(x) - \mathrm{d}\widetilde{B}_j(y)\| & = \|C_j^{-1}(\mathrm{d}B_j(Ax) - \mathrm{d}B_j(Ay))A\|  \leq C\kappa\|C_j^{-1}\||x-y|^\beta,
\end{align*}
where the constant $C$ depends on at most $d$. To show that we may
choose the neighbourhood $U$ and the constant in the claimed uniform
manner we need to show that suitable upper bounds hold for the norms
of $A^{-1}$ and each $C_j^{-1}$.

For $A^{-1}$, we note that
\begin{equation*}
\star \bigwedge_{j=1}^m \star X_j(\mathrm{d}B_j(x_0)) = \prod_{j=1}^m \|X_j(\mathrm{d}B_j(x_0))\|_{\Lambda^{d_j}(\mathbb{R}^d)} \star \bigwedge_{j=1}^m \bigwedge_{k \in \mathcal{K}_j} a_k
\end{equation*}
by \eqref{e:mainstar} and therefore
\begin{equation} \label{e:detAidentity}
  \star \bigwedge_{j=1}^m \star X_j(\mathrm{d}B_j(x_0)) = \det(A) \prod_{j=1}^m \|X_j(\mathrm{d}B_j(x_0))\|_{\Lambda^{d_j}(\mathbb{R}^d)}.
\end{equation}
Since $\|B_j\|_{C^{1,\beta}} \leq \kappa$ it follows that
\begin{equation*}
  \left|\star \bigwedge_{j=1}^m \star X_j(\mathrm{d}B_j(x_0))\right| \leq C |\det(A)|
\end{equation*}
for some constant $C$ depending on $\kappa$ and $d$. Since each
column of $A$ is a unit vector, it follows that the norm of $A^{-1}$
is bounded above by a constant depending on $\varepsilon,\kappa$ and
$d$.

For $C_j^{-1}$, from \eqref{e:mainstar} we get
\begin{equation} \label{e:detCjidentity}
  |\det(C_j)| = \|X_j(\mathrm{d}B_j(x_0))\|_{\Lambda^{d_j}(\mathbb{R}^d)} |\det(A)|.
\end{equation}
By \eqref{e:detAidentity},
\begin{equation*}
  \varepsilon \leq C\|X_j(\mathrm{d}B_j(x_0))\|_{\Lambda^{d_j}(\mathbb{R}^{d_j})}|\det(A)|,
\end{equation*}
for some constant $C$ depending on $\kappa$ and $d$. It follows that
the norm of $C_j^{-1}$ is also bounded above by a constant depending
on $\varepsilon,\kappa$ and $d$.

Applying Proposition \ref{p:straightBj} it follows that there exists
a neighbourhood $U$ of $x_0$ depending on at most $\beta,
\varepsilon, \kappa$ and $d$ such that
\begin{align*}
\int_{\mathbb{R}^d} \prod_{j=1}^m f_j(B_j(x))^{\frac{1}{m-1}}\psi(x)\,\mathrm{d}x \leq C|\det(A)| \prod_{j=1}^m \left(
\int_{\mathbb{R}^{d_j}} \widetilde{f}_j \right)^{\frac{1}{m-1}},
\end{align*}
where $C$ depends on at most $d$ and $\psi$. Thus
\begin{align*}
\int_{\mathbb{R}^d} \prod_{j=1}^m f_j(B_j(x))^{\frac{1}{m-1}}\psi(x)\,\mathrm{d}x & \leq C\frac{|\det(A)|}{\left(\prod_{j=1}^m |\det(C_j)|\right)^{\frac{1}{m-1}}} \prod_{j=1}^m \left( \int_{\mathbb{R}^{d_j}}
f_j \right)^{\frac{1}{m-1}} \\
& = C\left|\star \bigwedge_{j=1}^m \star X_j(\mathrm{d}B_j(x_0)) \right|^{-\frac{1}{m-1}} \prod_{j=1}^m \left( \int_{\mathbb{R}^{d_j}}
f_j \right)^{\frac{1}{m-1}},
\end{align*}
where the equality holds because of \eqref{e:detAidentity} and
\eqref{e:detCjidentity}. Theorem \ref{t:main} now follows.

For the various constants appearing in the above proof, one may
easily obtain some explicit dependence in terms of the relevant
parameters. Combined with Proposition \ref{p:straightBj}, this gives
additional information on the sizes of the neighbourhood $U$ and
constant $C$ appearing in the statement of Theorem \ref{t:main}. We
do not pursue this matter further here.

\section{Proof of Proposition \ref{p:straightBj}: Induction-on-scales} \label{section:induction}

Before stating the main induction lemma we use to prove Proposition
\ref{p:straightBj}, we need to fix some further notation. For each
$1 \leq j \leq m$ and $M > 0$, let $L^{1}_{M}(\mathbb{R}^{d_j})$
denote those nonnegative $f \in L^{1}(\mathbb{R}^{d_j})$ satisfying
$f(y_1) \leq 2 f(y_2)$ whenever $y_1$ and $y_2$ are in the support
of $f$ and $|y_1-y_2| \leq M^{-1}$; that is, those $f$ which are
effectively constant at the scale $M^{-1}$. One may easily check
that if $\mu$ is a finite measure on $\mathbb{R}^{d_j}$ then
$P_{c/M}^{(d_j)}*\mu\in L^1_M(\mathbb{R}^{d_j})$, where
$P_{c/M}^{(d_j)}$ denotes the Poisson kernel on $\mathbb{R}^{d_j}$
at height $c/M$. Here $c$ is a suitably large constant depending
only on $d_j$. By an elementary density argument, it will be enough
to prove Proposition \ref{p:straightBj} for $f_j\in
L^1_M(\mathbb{R}^{d_j})$, $1\leq j\leq m$, with neighbourhood $U$
and constant $C$ independent of $M$. As we shall shortly see, we
consider such a subclass of functions in order to provide a ``base
case" for the inductive argument.

For $\beta, \kappa > 0$, $1 < \alpha_0 < \alpha_1 < 1 + \beta$ and
$x_0 \in \mathbb{R}^d$ we let
$\mathcal{B}(\beta,\kappa,\alpha_0,\alpha_1,x_0)$ be the family of
data $\mathbf{B}$ such that $B_j$ belongs to
$C^{1,\beta}(Q(x_0,\delta_0))$ with $\|B_j\|_{C^{1,\beta}} \leq
\kappa$ and satisfies $\mathrm{d}B_j(x_0) = \Pi_j$, $1 \leq j \leq
m$. Here, $\delta_0$ is given by \eqref{e:R0defn}.

Now let $C(\delta,M)$ denote the best constant in the inequality
\begin{equation*} \label{e:FinnerQ}
  \int_Q \prod_{j=1}^m f_j(B_j(x))^{\frac{1}{m-1}}\,\mathrm{d}x \leq C \prod_{j=1}^m \left( \int_{\mathbb{R}^{d_j}} f_j \right)^{\frac{1}{m-1}}
\end{equation*}
over all $\mathbf{B} \in
\mathcal{B}(\beta,\kappa,\alpha_0,\alpha_1,x_0)$, all axis-parallel
subcubes $Q$ of $Q(x_0,\delta_0)$ with sidelength equal to $\delta$
and all inputs $\textbf{f}$ such that $f_j$ belongs to
$L^1_M(\mathbb{R}^{d_j})$, $1 \leq j \leq m$.

We note that the constant $C(\delta,M)$ also depends on the
parameters $\beta, \kappa, \alpha_0$ and $\alpha_1$, although there
is little to be gained in what follows from making this dependence
explicit. The main induction-on-scales lemma is the following.
\begin{lemma} \label{l:main}
For all $0 < \delta \leq \delta_0$ we have
\begin{equation*}
  C(\delta,M) \leq (1+10^d\delta^{\frac{\alpha_1-\alpha_0}{m-1}})C(2\delta^{\alpha_0},M).
\end{equation*}
\end{lemma}
The proof of Lemma \ref{l:main} is a little lengthy. Before giving
the proof we show how Lemma \ref{l:main} implies Proposition
\ref{p:straightBj}.

\subsection*{Deduction of Proposition \ref{p:straightBj} from Lemma \ref{l:main}}
Firstly we claim that the ``base case" inequality
\begin{equation} \label{e:basecase}
C(\delta_0/2^N,M) \leq 10^d
\end{equation}
holds for sufficiently large $N$. To see \eqref{e:basecase}, suppose
$\mathbf{B} \in \mathcal{B}(\beta,\kappa,\alpha_0,\alpha_1,x_0)$,
$Q$ is a subcube of $Q(x_0,\delta_0)$ with centre $x_Q$ and
sidelength $\delta_0/2^N$, and the input $\textbf{f}$ is such that
$f_j$ belongs to $L^1_M(\mathbb{R}^{d_j})$, $1 \leq j \leq m$. For
any $x \in Q$,
\begin{equation*}
  |B_j(x) - (B_j(x_Q) + \mathrm{d}B_j(x_Q)(x-x_Q))| \leq \kappa|x-x_Q|^{1+\beta} \leq 1/M
\end{equation*}
if $N$ is sufficiently large (depending on $\beta,\kappa,d$ and
$M$). Since $f_j \in L^{1}_M(\mathbb{R}^{d_j})$ it follows that
\begin{equation*}
  \int_Q \prod_{j=1}^m f_j(B_j(x))^{\frac{1}{m-1}}\,\mathrm{d}x \leq 2^m \int_{Q-\{x_Q\}} \prod_{j=1}^m f_j( \, \cdot + B_j(x_Q))(\mathrm{d}B_j(x_Q)x)^{\frac{1}{m-1}}\,\mathrm{d}x.
\end{equation*}
Now
\begin{equation*}
\|\mathrm{d}B_j(x_Q) - \Pi_j\| = \|\mathrm{d}B_j(x_Q) - \mathrm{d}B_j(x_0)\| \leq \frac{1}{100^d},
\end{equation*}
which implies that
\begin{equation*}
  \star\bigwedge_{j=1}^m\star X_j(\mathrm{d}B_j(x_Q))\geq \frac{1}{2},
\end{equation*}
and therefore
\begin{equation*}
  \int_Q \prod_{j=1}^m f_j(B_j(x))^{\frac{1}{m-1}}\,\mathrm{d}x \leq 10^d \prod_{j=1}^m \bigg( \int_{\mathbb{R}^{d_j}} f_j \bigg)^{\frac{1}{m-1}}
\end{equation*}
by Proposition \ref{p:linearcase}. Hence, \eqref{e:basecase} holds.

For $0 < \delta \leq \delta_0 \leq (1/4)^{1/\alpha_0-1}$ it follows
from Lemma \ref{l:main} that
\begin{equation} \label{e:induction}
  C(\delta,M) \leq (1+ 10^d \delta^{\frac{\alpha_1-\alpha_0}{m-1}})C(\delta/2,M).
\end{equation}
Applying \eqref{e:induction} iteratively $N$ times we see that
\begin{equation*}
  C(\delta_0,M) \leq C(\delta_0/2^N,M)\prod_{r=0}^{N-1} \big(1+10^d(\delta_0/2^r)^{\frac{\alpha_1-\alpha_0}{m-1}}\big).
\end{equation*}
The product term is under control uniformly in $N$ because
\begin{align*}
  \log \prod_{r=0}^{N-1} \big(1+10^d(\delta_0/2^r)^{\frac{\alpha_1-\alpha_0}{m-1}}\big) & =
\sum_{r=0}^{N-1} \log \big(1+10^d(\delta_0/2^r)^{\frac{\alpha_1-\alpha_0}{m-1}}\big)  \\
  & \leq 10^d\delta_0^{\frac{\alpha_1-\alpha_0}{m-1}} \sum_{r=0}^\infty 2^{-\frac{\alpha_1-\alpha_0}{m-1}r} \\
  & \leq \frac{10^d\delta_0^{\frac{\alpha_1-\alpha_0}{m-1}}}{1-2^{-\frac{\alpha_1-\alpha_0}{m-1}}}.
\end{align*}
From the base case \eqref{e:basecase} it follows that
\begin{equation*}
  C(\delta_0,M) \leq 10^d \exp\left(\frac{10^d\delta_0^{\frac{\alpha_1-\alpha_0}{m-1}}}{1-2^{-\frac{\alpha_1-\alpha_0}{m-1}}}\right);
\end{equation*}
that is,
\begin{equation} \label{e:lastone}
   \int_{Q(x_0,\delta_0)} \prod_{j=1}^m f_j(B_j(x))^{\frac{1}{m-1}}\,\mathrm{d}x \leq 10^d \exp\left(\frac{10^d\delta_0^{\frac{\alpha_1-\alpha_0}{m-1}}}{1-2^{-\frac{\alpha_1-\alpha_0}{m-1}}}\right) \prod_{j=1}^m \left( \int_{\mathbb{R}^{d_j}} f_j \right)^{\frac{1}{m-1}}
\end{equation}
for all $f_j \in L^1_M(\mathbb{R}^{d_j})$, $1 \leq j \leq m$. Since
the constant in \eqref{e:lastone} is independent of $M$, it follows
that the inequality is valid for all $f_j \in
L^1(\mathbb{R}^{d_j})$. This completes our proof of Proposition
\ref{p:straightBj}.

\subsection*{Proof of Lemma \ref{l:main}}
Suppose $\mathbf{B} = (B_j) \in
\mathcal{B}(\beta,\kappa,\alpha_0,\alpha_1,x_0)$, $Q$ is an
axis-parallel subcube of $Q(x_0,\delta_0)$ with sidelength equal to
$\delta$ and centre $x_Q$, and suppose $\textbf{f} = (f_j)$ is such
that $f_j$ belongs to $L^{1}_M(\mathbb{R}^{d_j})$, $1\leq j\leq m$.
Notice that the desired inequality
\begin{equation} \label{e:mainestimate}
  \int_{Q} \prod_{j=1}^m f_j(B_j(x))^{\frac{1}{m-1}}\,\mathrm{d}x
  \leq (1+10^d\delta^{\frac{\alpha_1-\alpha_0}{m-1}})C(2\delta^{\alpha_0},M) \prod_{j=1}^m \left( \int_{\mathbb{R}^{d_j}}f_j \right)^{\frac{1}{m-1}}
\end{equation}
is invariant under the transformation $(\mathbf{B},\mathbf{f},Q)
\mapsto
(\widetilde{\mathbf{B}},\widetilde{\mathbf{f}},\widetilde{Q})$ where
$\widetilde{B}_j = B_j(\,\cdot+x_{Q}) - B_j(x_{Q})$, $\widetilde{Q}
= Q - \{x_{Q}\}$ and $\widetilde{f}_j = f_j( \,\cdot + B_j(x_{Q}))$.
Hence, without loss of generality, $Q = Q(0,\delta)$ and $B_j(0) = 0$ for
$1 \leq j \leq m$. This reduction is merely for notational
convenience; in particular, it ensures
$$
|B_j(x) - \mathrm{d}B_j(0)x| \leq \kappa|x|^{1+\beta}.
$$
By the smoothness hypothesis, we have that
\begin{equation} \label{e:Pi_j}
\|\mathrm{d}B_j(0) - \Pi_j \|  \leq \frac{1}{100^d}
\end{equation}
for sufficiently small $c_{d}$. Since
\begin{equation*}
\ker \Pi_j = \langle \{e_k:k \in \mathcal{K}_j\} \rangle,
\end{equation*}
it follows that for each $1\leq k\leq d$ there exist $a_k \in
\mathbb{R}^d$ such that
\begin{equation} \label{e:tildea_k}
|a_k - e_k| \leq \frac{1}{10^d},
\end{equation} and
\begin{equation*}
\ker \mathrm{d}B_j(0) = \langle \{a_k : k \in \mathcal{K}_j\}\rangle
\end{equation*}
for each $1\leq j\leq m$. Here, $e_k$ denotes the $k$th standard
basis vector in $\mathbb{R}^d$.

The proof of Lemma \ref{l:main} naturally divides into four steps.

\vspace{0.5cm} \textbf{Step I: Foliations of $\mathbb{R}^d$}

For each $1 \leq i \leq d$ consider the one-parameter family of hypersurfaces
\begin{equation} \label{e:Rdhyperplanes0}
  \langle \{a_k : k \neq i \} \rangle + \bigg\{s \star \bigwedge_{k \neq i} a_k\bigg\}
\end{equation}
where $s \in \mathbb{R}$. We point out that $\star \bigwedge_{k \neq i} a_k$ is simply the
cross product of the vectors $\{ a_k : k \neq i \}$, yielding a
vector normal to $\langle \{ a_k : k \neq i \} \rangle$. The set of vectors $\{ \star \bigwedge_{k \neq i} a_k:1\leq
i\leq d\}$ in $\mathbb{R}^d$ is linearly independent since the same
is true of $\{a_i : 1 \leq i \leq d\}$. Consequently, we may decompose
$\mathbb{R}^d$ into parallelepipeds whose faces are contained in
hyperplanes of the form \eqref{e:Rdhyperplanes0}, $1 \leq i \leq d$.
We will use this to decompose the cube $Q$. As we shall see in the steps that
follow, an important feature of these hypersurfaces is that they may be expressed
as inverse images of hypersurfaces under the mappings $\mathrm{d}B_j(0)$.
To this end, let $\sigma : \{1,\ldots,d\} \rightarrow \{1,\ldots,m\}$ be the map
given by
\begin{equation*}
  \sigma(i) = \text{$(j+1)$ mod $m$}
\end{equation*}
for $i \in \mathcal{K}_j$. As will become apparent under closer inspection, there is some
freedom in our choice of this map; all that we require of $\sigma$
is that $j\mapsto\sigma(\mathcal{K}_j)$ is a permutation of
$\{1,2,\hdots,m\}$ with no fixed points.

For each $1 \leq i \leq d$ and $J \subset \mathbb{R}$ we define the
set
\begin{equation} \label{e:SigmaiJdefn}
  \Sigma(i,J) = \mathrm{d}B_{\sigma(i)}(0)\langle \{a_k : k \neq i \} \rangle
  + \bigg\{s\; \mathrm{d}B_{\sigma(i)}(0) \bigg(\star \bigwedge_{k \neq i} a_k\bigg) : s \in J \bigg\}.
\end{equation}
If $J = \{s\}$ is a singleton set then
\begin{equation*}
  \Sigma(i,\{s\}) = \mathrm{d}B_{\sigma(i)}(0)\langle \{a_k : k \neq i \} \rangle
  + \bigg\{s\; \mathrm{d}B_{\sigma(i)}(0) \bigg(\star \bigwedge_{k \neq i} a_k\bigg)\bigg\}
\end{equation*}
is a hyperplane in $\mathbb{R}^{d_{\sigma(i)}}$ since $\ker
\mathrm{d}B_{\sigma(i)}(0) \subseteq \langle \{a_k : k \neq i \}
\rangle$. Similarly,
\begin{equation} \label{e:Rdhyperplanes}
  \mathrm{d}B_{\sigma(i)}(0)^{-1}\Sigma(i,\{s\}) = \langle \{a_k : k \neq i \} \rangle + \bigg\{s \star \bigwedge_{k \neq i} a_k\bigg\}
\end{equation}
which is of course the hyperplane \eqref{e:Rdhyperplanes0}.

As outlined in Section
\ref{section:proofidea}, a regular decomposition of $\mathbb{R}^d$
into parallelepipeds of equal size and adapted to a lattice (where
for each $i$, the sequence of parameters $s^{(i)}$ that we choose is
in arithmetic progression) will not suffice to prove Lemma
\ref{l:main}. Moreover, our decomposition will need to incorporate
certain ``buffer zones" between the parallelepipeds to create
separation. In Step II below we determine the location of the buffer
zones and thus the desired decomposition of $Q$.

\vspace{0.5cm} \textbf{Step II: The decomposition of $Q$}

For each $1 \leq i \leq d$ we claim that there exists a sequence
$(s^{(i)}_n)_{n \geq 1}$ such that
\begin{equation} \label{e:pigeonpara}
s^{(i)}_n + \tfrac{1}{2}\delta^{\alpha_0} \leq s^{(i)}_{n+1} \leq s^{(i)}_n + \delta^{\alpha_0}
\end{equation}
and
\begin{equation} \label{e:pigeonnorm}
\int_{\Sigma(i,[s^{(i)}_{n+1},s^{(i)}_{n+1}+\delta^{\alpha_1}])} f_{\sigma(i)}\chi_{Q} \leq
4\delta^{\alpha_1 - \alpha_0} \int_{\Sigma(i,[s^{(i)}_{n} + \frac{1}{2}\delta^{\alpha_0}, s^{(i)}_{n} + \delta^{\alpha_0}])} f_{\sigma(i)}\chi_{Q}.
\end{equation}
To prove this, we shall choose the sequence $(s^{(i)}_n)_{n \geq 1}$
iteratively. We begin by choosing $s^{(i)}_1$ to be any real number
such that $B_{\sigma(i)}(Q) \subseteq \Sigma(i,[s^{(i)}_1,\infty))$.
Suppose that we have chosen $s^{(i)}_1,\ldots,s^{(i)}_n$ for some $n
\geq 1$. Now let $N$ be the largest integer which is less than or
equal to $\tfrac{1}{2}\delta^{\alpha_0-\alpha_1}$. Set
$\zeta^{(i)}_0 = s_n^{(i)} + \frac{1}{2}\delta^{\alpha_0}$ and then
define $\zeta^{(i)}_r = \zeta^{(i)}_{r-1} + \delta^{\alpha_1}$
iteratively for $1 \leq r \leq N$ so that
\begin{equation*}
[s^{(i)}_n + \tfrac{1}{2}\delta^{\alpha_0},s^{(i)}_n + \delta^{\alpha_0}] \supseteq
[s^{(i)}_{n} + \tfrac{1}{2}\delta^{\alpha_0}, s^{(i)}_{n} + \tfrac{1}{2}\delta^{\alpha_0} + N\delta^{\alpha_1}] = \bigcup_{r=1}^N
[\zeta^{(i)}_{r-1},\zeta^{(i)}_{r}].
\end{equation*}
Then,
\begin{align*}
\int_{\Sigma(i,[s^{(i)}_{n} + \frac{1}{2}\delta^{\alpha_0}, s^{(i)}_{n} + \delta^{\alpha_0}])} f_{\sigma(i)}\chi_{Q} \geq
\sum_{r=1}^N \int_{\Sigma(i,[\zeta^{(i)}_{r-1},\zeta^{(i)}_r])} f_{\sigma(i)}\chi_{Q},
\end{align*}
and therefore by the choice of $\delta_0$ in \eqref{e:R0defn} and
the pigeonhole principle, there exists $s^{(i)}_{n+1}$ such that
\eqref{e:pigeonpara} holds and
\begin{equation*}
\int_{\Sigma(i,[s^{(i)}_{n} + \frac{1}{2}\delta^{\alpha_0}, s^{(i)}_{n} + \delta^{\alpha_0}])} f_{\sigma(i)}\chi_{Q} \geq
\tfrac{1}{4}\delta^{\alpha_0-\alpha_1} \int_{\Sigma(i,[s^{(i)}_{n+1},s^{(i)}_{n+1}+\delta^{\alpha_1}])} f_{\sigma(i)}\chi_{Q};
  \end{equation*}
that is, \eqref{e:pigeonnorm} also holds.

We shall use the notation $J(i,n,0)$ and $J(i,n,1)$ for the
intervals given by
\begin{equation} \label{e:J0defn}
  J(i,n,0) = (s^{(i)}_{n} + \tfrac{2}{3}\delta^{\alpha_1}, s^{(i)}_{n+1} + \tfrac{1}{3}\delta^{\alpha_1}]
\end{equation}
and
\begin{equation} \label{e:J1defn}
  J(i,n,1) = (s^{(i)}_{n} + \tfrac{1}{3}\delta^{\alpha_1},s^{(i)}_{n} + \tfrac{2}{3}\delta^{\alpha_1}].
\end{equation}
Notice that the lengths of $J(i,n,0)$ and $J(i,n,1)$ are comparable to
$\delta^{\alpha_0}$ and $\delta^{\alpha_1}$ respectively.

By construction, the sets $\Sigma(i,J(i,n,1))$ contain a relatively small amount of the mass of the function $f_{\sigma(i)}$ in the sense of \eqref{e:pigeonnorm}. Furthermore, the inverse images of these sets,
\begin{equation}\label{buffdef}\mathrm{d}B_{\sigma(i)}(0)^{-1} \Sigma(i,J(i,n,1)),\end{equation}
are $O(\delta^{\alpha_1})$ neighbourhoods of hyperplanes in $\mathbb{R}^d$, which as $n$ varies are separated by $O(\delta^{\alpha_0})$. We refer to the sets \eqref{buffdef} as buffer zones.

The decomposition of $Q$ we use is given by
\begin{equation} \label{e:Qdec}
  Q = \bigcup_{\chi \in \{0,1\}^d} \bigcup_{n \in \mathbb{N}^d} P(n,\chi)
\end{equation}
where
\begin{equation} \label{e:Pkdefn}
  P(n,\chi) = \bigcap_{i=1}^d \mathrm{d}B_{\sigma(i)}(0)^{-1} \Sigma(i,J(i,n_i,\chi_i))  \cap Q.
\end{equation}
When $\chi = 0$, the $P(n,\chi)$ are large parallelepipeds
(intersected with $Q$) with sidelength approximately
$\delta^{\alpha_0}$ which form the main part of our decomposition.
For $\chi \neq 0$, the $P(n,\chi)$ are small parallelepipeds
(intersected with $Q$) with at least one sidelength approximately
$\delta^{\alpha_1}$, which decompose the buffer zones.

\vspace{0.5cm} \textbf{Step III: Disjointness}

In this step we make precise the role of the buffer zones. For each
$1 \leq j \leq m$, $\ell \in \mathbb{N}^{d_j}$ and $\chi \in
\{0,1\}^d$ let
\begin{equation*}
T_j(\ell,\chi) = \bigcup_{\substack{n \in \mathbb{N}^d: \\ \Pi_jn = \ell}} P(n,\chi).
\end{equation*}
It is the disjointness of the images of such sets under the mapping
$B_j$ that is crucial to the induction-on-scales argument which
follows in Step IV.


\begin{proposition} \label{p:disjoint}
Fix $j$ with $1 \leq j \leq m$ and $\chi \in \{0,1\}^d$. If
$\ell,\ell' \in \mathbb{N}^{d_j}$ are distinct then
\begin{equation} \label{e:disjointP}
B_j(T_j(\ell,\chi)) \cap B_j(T_j(\ell',\chi)) = \emptyset.
\end{equation}
\end{proposition}
To prove Proposition \ref{p:disjoint} we use the following.
\begin{lemma} \label{l:Phi}
For each $1 \leq j \leq m$ there exists a map $\Phi_j : \mathbb{R}^d
\rightarrow \mathbb{R}^d$ such that
\begin{enumerate}
\item[(i)]  $\Phi_j(0) = 0$ and $\mathrm{d}\Phi_j(0)$ is equal to the
identity matrix $I_d$,

\item[(ii)]  $B_j = \mathrm{d}B_j(0) \circ \Phi_j$,

\item[(iii)]  $\|\mathrm{d}\Phi_j(x) - \mathrm{d}\Phi_j(y)\| \leq
2\kappa|x-y|^\beta$ for each $x,y \in Q$,

\item[(iv)] $|x - \Phi_j(x)| \leq 2d\kappa \delta^{1+\beta}$ for
each $x \in Q$.
\end{enumerate}
\end{lemma}

\begin{proof}
Let $\widetilde{I}_{d_j}$ be the invertible $d_j \times d_j$ matrix
obtained by deleting the $k$th column of $\mathrm{d}B_j(0)$ for each
$k \in \mathcal{K}_j$. For $k \in \mathcal{K}_j$ define the $k$th
component of $\Phi_j(x)$ to be $x_k$. Define the remaining $d_j$
components of $\Phi_j(x)$ by stipulating that the element of
$\mathbb{R}^{d_j}$ obtained by deleting the $k$th components of
$\Phi_j(x)$ for $k \in \mathcal{K}_j$ is equal to
\begin{equation*}
\widetilde{I}_{d_j}^{-1}\bigg( B_j(x) - \sum_{k \in \mathcal{K}_j} x_k \mathrm{d}B_j(0)e_k \bigg).
\end{equation*}
Then a direct computation verifies that Properties (i) and (ii) hold
for $\Phi_j$. Also,
\begin{equation*}
\|\mathrm{d}\Phi_j(x) - \mathrm{d}\Phi_j(y)\| = \| \widetilde{I}_{d_j}^{-1}(\mathrm{d}B_j(x) - \mathrm{d}B_j(y))\| \leq 2\kappa|x-y|^\beta,
\end{equation*}
since $\|\widetilde{I}_{d_j} - I_{d_j}\| \leq 1/10$, and therefore
(iii) holds. Finally, Property (iv) follows from Properties (i) and
(iii), and the mean value theorem.
\end{proof}

\begin{proof}[Proof of Proposition \ref{p:disjoint}]
Suppose $\ell \neq \ell'$ and, for a contradiction, suppose that $z
= B_j(x) = B_j(y)$ where $x \in T_j(\ell,\chi)$ and $y \in
T_j(\ell',\chi)$. Then $x \in P(n,\chi)$ and $y \in P(n',\chi)$ for
some $n,n' \in \mathbb{N}^{d}$ satisfying $\Pi_jn = \ell$ and
$\Pi_jn' = \ell'$. Since $\Pi_jn \neq \Pi_jn'$ there exists $i \in
\mathcal{K}_j^c$ such that $n_i \neq n_i'$.

By \eqref{e:Pkdefn} and \eqref{e:SigmaiJdefn} it follows that there
exist $s(x) \in J(i,n_i,\chi_i)$ and $s(y) \in J(i,n_i',\chi_i)$
such that
\begin{equation*}
\bigg \langle x,\star \bigwedge_{k \neq i} a_k \bigg \rangle =
s(x)\bigg|\star \bigwedge_{k \neq i} a_k \bigg|^2
\qquad \text{and} \qquad
\bigg \langle y,\star \bigwedge_{k \neq i} a_k \bigg \rangle =
s(y)\bigg|\star \bigwedge_{k \neq i} a_k \bigg|^2.
\end{equation*}
Therefore
\begin{equation*}
  \bigg| \bigg\langle x-y, \star\bigwedge_{k \neq i} a_k \bigg\rangle \bigg| =
|s(x) - s(y)| \bigg| \star\bigwedge_{k \neq i} a_k \bigg|^2 \geq \tfrac{1}{3}\delta^{\alpha_1}
\bigg| \star\bigwedge_{k \neq i} a_k \bigg|^2
\end{equation*}
where the inequality follows from \eqref{e:pigeonpara},
\eqref{e:J0defn} and \eqref{e:J1defn} since $n_i \neq n_i'$.

On the other hand, since $x$ and $y$ belong to the fibre
$B_j^{-1}(z)$, it follows from Lemma \ref{l:Phi}(ii) that $\Phi_j(x)$
and $\Phi_j(y)$ belong to $\mathrm{d}B_j(0)^{-1}(z)$ and thus
$\Phi_j(x) - \Phi_j(y) \in \ker \mathrm{d}B_j(0)$. Since $i \in \mathcal{K}_j^c$ and $\ker \mathrm{d}B_j(0) = \langle
\{ a_r : r \in \mathcal{K}_j\}\rangle$ the vector $\star
\bigwedge_{k \neq i} a_k$ belongs to the orthogonal complement of
$\ker \mathrm{d}B_j(0)$. Therefore,
\begin{align*}
\bigg\langle x-y, \star \bigwedge_{k \neq i} a_k \bigg\rangle =  \bigg\langle x - \Phi_j(x), \star \bigwedge_{k \neq i} a_k \bigg \rangle - \bigg\langle y -
\Phi_j(y), \star \bigwedge_{k \neq i} a_k \bigg\rangle,
\end{align*}
and so by the Cauchy--Schwarz inequality and Lemma \ref{l:Phi}(iv) it
follows that
\begin{equation*}
  \bigg|\bigg\langle x-y, \star \bigwedge_{k \neq i} a_k \bigg\rangle\bigg| \leq
4d\kappa \delta^{1+\beta}\bigg| \star \bigwedge_{k \neq i} a_k \bigg|.
\end{equation*}
Since $| \star \bigwedge_{k \neq i} a_k| \geq 1/2$ we conclude that
$24d\kappa \delta^{1+\beta} \geq \delta^{\alpha_1}$. For a
sufficiently small choice of $c_d$, this is our desired
contradiction.
\end{proof}

\vspace{0.5cm} \textbf{Step IV: The conclusion via the discrete inequality}

Using the decomposition in Step II,
\begin{equation*}
  \int_{Q} \prod_{j=1}^m f_j(B_j(x))^{\frac{1}{m-1}}\,\mathrm{d}x =
\sum_{\chi \in \{0,1\}^d} \sum_{n \in \mathbb{N}^d} \int_{P(n,\chi)} \prod_{j=1}^m f_j(B_j(x))^{\frac{1}{m-1}}\,\mathrm{d}x.
\end{equation*}
By \eqref{e:tildea_k},
\begin{equation*}
\bigg|\star \bigwedge_{k \neq i} a_k - e_i \bigg| \leq \frac{1}{10},
\end{equation*}
and thus each $P(n,\chi)$ is contained in an axis-parallel cube with
sidelength equal to $2\delta^{\alpha_0}$.

\subsection*{The main term: $\chi = 0$} It follows that
\begin{align*}
\sum_{n \in \mathbb{N}^d} \int_{P(n,0)} \prod_{j=1}^m f_j(B_j(x))^{\frac{1}{m-1}}\,\mathrm{d}x
& \leq C(2\delta^{\alpha_0},M) \sum_{n \in \mathbb{N}^d} \prod_{j=1}^m \left( \int_{B_j(P(n,0))} f_j \right)^{\frac{1}{m-1}} \\
& \leq C(2\delta^{\alpha_0},M) \sum_{n \in \mathbb{N}^d} \prod_{j=1}^m F_j(\Pi_jn)^{\frac{1}{m-1}}
\end{align*}
where $F_j : \mathbb{N}^{d_j} \rightarrow [0,\infty)$ is given by
\begin{equation*}
F_j(\ell) = \int_{B_j(T_j(\ell,0))} f_j.
\end{equation*}
Hence, by \eqref{e:discreteFinner},
\begin{equation*}
  \sum_{n \in \mathbb{N}^d} \int_{P(n,0)} \prod_{j=1}^m f_j(B_j(x))^{\frac{1}{m-1}}\,\mathrm{d}x \leq
C(2\delta^{\alpha_0},M) \prod_{j=1}^m \| F_j \|_{\ell^1(\mathbb{N}^{d_j})}^{\frac{1}{m-1}}.
\end{equation*}
Consequently, by Proposition \ref{p:disjoint},
\begin{equation} \label{e:Lambda0}
  \sum_{n \in \mathbb{N}^d} \int_{P(n,0)} \prod_{j=1}^m f_j(B_j(x))^{\frac{1}{m-1}}\,\mathrm{d}x
\leq C(2\delta^{\alpha_0},M) \prod_{j=1}^m \left( \int_{\mathbb{R}^{d_j}} f_j \right)^{\frac{1}{m-1}}.
\end{equation}
\subsection*{The remaining terms: $\chi \neq 0$}
To allow us to capitalise on the pigeonholing in Step II we need the
following.
\begin{lemma} \label{l:slabs} For each $1 \leq i \leq d$ we have
\begin{equation*}
\mathrm{d}B_{\sigma(i)}(0)^{-1} \Sigma(i,J(i,n_i,1)) \cap Q \subseteq B_{\sigma(i)}^{-1}
\Sigma(i,[s^{(i)}_{n_i}, s^{(i)}_{n_i} + \delta^{\alpha_1}])  \cap Q.
\end{equation*}
\end{lemma}
Note here that $[s^{(i)}_{n_i}, s^{(i)}_{n_i} + \delta^{\alpha_1}]$ is simply the ``concentric triple"
of $J(i,n_i,1)$.
\begin{proof}
Suppose $x \in Q$ satisfies $\mathrm{d}B_{\sigma(i)}(0)x \in
\Sigma(i,J(i,n_i,1))$ so that
\begin{equation} \label{e:dBj0x}
\mathrm{d}B_{\sigma(i)}(0)x = \mathrm{d}B_{\sigma(i)}(0)y + s\mathrm{d}B_{\sigma(i)}(0) \bigg( \star \bigwedge_{k \neq i} a_k \bigg)
\end{equation}
for some $s \in [s^{(i)}_{n_i} + \tfrac{1}{3}\delta^{\alpha_1},
s^{(i)}_{n_i} + \tfrac{2}{3}\delta^{\alpha_1}]$ and $y \in \langle
\{a_k : k \neq i\} \rangle$, by \eqref{e:J1defn} and
\eqref{e:SigmaiJdefn}. By Lemma \ref{l:Phi}(ii),
\begin{equation} \label{e:Bjx}
B_{\sigma(i)}(x) = \mathrm{d}B_{\sigma(i)}(0)x + \mathrm{d}B_{\sigma(i)}(0)(\Phi_{\sigma(i)}(x) - x).
\end{equation}
Now $\Phi_{\sigma(i)}(x) - x = y' + s' \star
\bigwedge_{k \neq i} a_k$ for some $s' \in \mathbb{R}$
and $y' \in \langle \{a_k : k \neq i\} \rangle$, and thus
$$
\bigg \langle \Phi_{\sigma(i)}(x) - x, \star \bigwedge_{k \neq i} a_k \bigg \rangle = s' \bigg|\star \bigwedge_{k \neq i} a_k \bigg|^2.
$$
Since $|\star \bigwedge_{k \neq i} a_k| \geq 1/2$, and by the
Cauchy--Schwarz inequality and Lemma \ref{l:Phi}(iv), it follows
that $ |s'| \leq 4d\kappa\delta^{1+\beta}$. Now $s +
s' \in [s^{(i)}_{n_i},s^{(i)}_{n_i}+\delta^{\alpha_1}]$
for a sufficiently small choice of $c_d$. Therefore, by
\eqref{e:dBj0x} and \eqref{e:Bjx}, $B_{\sigma(i)}(x) \in
\Sigma(i,[s^{(i)}_{n_i},s^{(i)}_{n_i}+\delta^{\alpha_1}])$ as
required.
\end{proof}
Fix $\chi \neq 0$ and any $i$ such that $\chi_i = 1$. As above for
the main term, it follows from \eqref{e:discreteFinner} that
\begin{equation*}
\sum_{n \in \mathbb{N}^d} \int_{P(n,\chi)} \prod_{j=1}^m f_j(B_j(x))^{\frac{1}{m-1}} \,\mathrm{d}x
\leq C(2\delta^{\alpha_0},M) \prod_{j=1}^m \|F_j\|_{\ell^1(\mathbb{N}^{d_j})}^{\frac{1}{m-1}}
\end{equation*}
where now
\begin{equation*}
F_j(\ell) = \int_{B_j(T_j(\ell,\chi))} f_j.
\end{equation*}
By Proposition \ref{p:disjoint} it follows that
\begin{equation*}
\sum_{n \in \mathbb{N}^d} \int_{P(n,\chi)} \prod_{j=1}^m f_j(B_j(x))^{\frac{1}{m-1}} \,\mathrm{d}x
\leq C(2\delta^{\alpha_0},M) \| F_{\sigma(i)} \|_{\ell^1(\mathbb{N}^{d_{\sigma(i)}})}^{\frac{1}{m-1}} \prod_{j \neq \sigma(i)}
\bigg( \int_{\mathbb{R}^{d_j}} f_j  \bigg)^{\frac{1}{m-1}}
\end{equation*}
and thus it suffices to show that
\begin{equation} \label{e:bitofgain}
\| F_{\sigma(i)} \|_{\ell^1(\mathbb{N}^{d_{\sigma(i)}})} \leq 4\delta^{\alpha_1-\alpha_0}.
\end{equation}

To see \eqref{e:bitofgain}, first set $j = \sigma(i)$. Given the
choice of notation in Step II, it is convenient to write
\begin{equation*}
\| F_j \|_{\ell^1(\mathbb{N}^{d_j})} = \sum_{\ell \in \mathbb{N}^{d_j}} \int_{B_j(T_j(\ell,\chi))} f_j =
\sum_{\substack{n_k : \\ k \in \mathcal{K}_{j}^c}} \int_{B_{j}(T_{j}(\Pi_{j}n,\chi))} f_{j}.
\end{equation*}
Now, since $i \in \mathcal{K}_{j}^c$ we may write
\begin{align*}
\| F_j \|_{\ell^1(\mathbb{N}^{d_j})} = \sum_{n_i} \sum_{\substack{n_k : \\ k \in \mathcal{K}_{j}^c \setminus \{i\}}}
\int_{B_{j}(T_{j}(\Pi_{j}n,\chi))} f_{j}.
\end{align*}
By Lemma \ref{l:slabs} it follows that
\begin{align*}
  \bigcup_{\substack{n_k : \\ k \in \mathcal{K}_{j}^c \setminus \{i\}}} B_{j}(T_{j}(\Pi_{j}n,\chi))
\subseteq \Sigma(i,[s^{(i)}_{n_i},s^{(i)}_{n_i}+\delta^{\alpha_1}])  \cap Q.
\end{align*}
Therefore, by Proposition \ref{p:disjoint} and \eqref{e:pigeonnorm},
\begin{align*}
\sum_{\substack{n_k : \\ k \in \mathcal{K}_{j}^c \setminus \{i\}}}
\int_{B_{j}(T_{j}(\Pi_{j}n,\chi))} f_{j} &
\leq \int_{\Sigma(i,[s^{(i)}_{n_i},s^{(i)}_{n_i}+\delta^{\alpha_1}])} f_{j}\chi_{Q} \\
& \leq 4\delta^{\alpha_1-\alpha_0}
\int_{\Sigma(i,[s^{(i)}_{n_i-1} + \frac{1}{2}\delta^{\alpha_0},s^{(i)}_{n_i-1} + \delta^{\alpha_0}])} f_{j}\chi_{Q},
\end{align*}
from which \eqref{e:bitofgain} follows by summing in $n_i$ and
disjointness. This completes the proof of Lemma \ref{l:main}.

\begin{remark}
  In Theorem \ref{t:main}, the smoothness assumption that each mapping $B_j$ belongs to $C^{1,\beta}$ may be weakened. Suppose that each $B_j$ is a $C^1$ submersion in a neighbourhood of $x_0$ such that the modulus of continuity of $\mathrm{d}B_j$, which we denote by $\omega_{\mathrm{d}B_j}$, satisfies
  \begin{equation*}
    \omega_{\mathrm{d}B_j}(\delta) \leq \kappa\Omega(\delta),
  \end{equation*}
  where, for some $0 < \eta < 1$, $\Omega$ satisfies the summability condition
  \begin{equation} \label{e:sum}
    \sum_{r=0}^\infty \Omega(2^{-r})^{1-\eta} < \infty
  \end{equation}
  and $\kappa$ is a positive constant.
  Without significantly altering the above proof, one can show that Theorem \ref{t:main} holds under such a smoothness hypothesis. Of course, Theorem \ref{t:main} corresponds to $\Omega(\delta) = \delta^\beta$ with $\beta>0$. It is of course easy to choose $\Omega$ satisfying $\delta^\beta = o(\Omega(\delta))$ as $\delta \rightarrow 0$ for all $\beta>0$, and still satisfying \eqref{e:sum}; for example, $\Omega(\delta) = (\log 1/\delta)^{-2}$. Naturally, one pays for allowing a lower level of smoothness in the size of the neighbourhood on which the estimate in \eqref{linadmisnl} holds.
\end{remark}

\section{Proof of Corollary \ref{lemma6}} \label{section:corproof}
Without loss of generality we may suppose that there is a point $a$
belonging to a sufficiently small neighbourhood of the origin in
$(\mathbb{R}^{d-1})^{d-1}$ (depending on at most $d, \beta,
\varepsilon$ and $\kappa$) such that $F(a)=0$; otherwise the
neighbourhood $V$ in the statement of the corollary could be chosen
so that the left-hand side of \eqref{lemma6est} vanishes. By
considering a translation taking $a$ to the origin, we may suppose
that $a=0$. (Here we are using the uniformity claim relating to the
neighbourhood $V$.)

Furthermore, we may assume that
\begin{equation} \label{e:orthhyp}
\nabla_{u_j}F(0) = e_j,
\end{equation}
the $j$th standard basis vector in $\mathbb{R}^{d-1}$, for each $1
\leq j \leq d-1$. We shall see that the full generality of Corollary
\ref{lemma6} follows from this case by a change of variables.

Fix nonnegative $f_j \in L^{(d-1)'}(\mathbb{R}^{d-1})$, $1 \leq j
\leq m$. We proceed in a similar way to the proof of Proposition 7
of \cite{BCW}. Since $\partial_{(u_{d-1})_{d-1}}F(0) = 1$ it follows
that there exists a neighbourhood $W$ of the origin in
$\mathbb{R}^{d(d-2)}$ and a mapping $\eta : W \rightarrow
\mathbb{R}$ such that for each
\begin{equation*}
x = (u_1,\ldots,u_{d-2},(u_{d-1})_1,\ldots,(u_{d-1})_{d-2}) \in W
\end{equation*}
we have
\begin{equation} \label{e:etadefn}
  F(x,\eta(x)) = 0.
\end{equation}
The neighbourhood $W$ depends only on $\beta$ and $\kappa$, and the
mapping $\eta$ satisfies $\|\eta\|_{C^{1,\beta}} \leq
\widetilde{\kappa}$ for some constant $\widetilde{\kappa}$ which
depends only on $d, \beta$ and $\kappa$. Our claims follow from the
implicit function theorem in quantitative form. For completeness we
have included an adequate version in Appendix \ref{appendix:IFT}.

Let $B_j : W \rightarrow \mathbb{R}^{d-1}$ be given by
\begin{equation*}
  B_j(x) = (x_{(d-1)j-d+2}, \ldots, x_{(d-1)j})
\end{equation*}
for $1 \leq j \leq d-2$,
\begin{equation*}
  B_{d-1}(x) = (x_{(d-1)^2-d+2}, \ldots, x_{(d-1)^2-1}, \eta(x)),
\end{equation*}
and
\begin{equation*}
  B_d = B_1 + \cdots + B_{d-1}.
\end{equation*}
We claim that there exists a neighbourhood $U$ of the origin, with
$U \subset W$, depending only on $d, \beta$ and $\kappa$, and a
constant $C$ depending on $d$, such that
\begin{equation} \label{e:NBL}
  \int_U \prod_{j=1}^d f_j(B_j(x))\,\mathrm{d}x \leq C \prod_{j=1}^d \|f_j\|_{(d-1)'}.
\end{equation}
Since the subspaces $\ker \mathrm{d}B_1(0), \ldots, \ker
\mathrm{d}B_d(0)$ are such that at least one pair has a nontrivial
intersection, we cannot directly apply Theorem \ref{t:main} to
$\mathbf{B} = (B_j)$ in order to prove \eqref{e:NBL} (except in the
special case $d=3$ -- see \cite{BCW}).
It is, however, possible to construct mappings $B_j^\oplus :
\mathbb{R}^{d(d-2)} \rightarrow \mathbb{R}^{(d-1)(d-2)}$ for $1 \leq
j \leq d$ in block form so that
\begin{equation} \label{e:directsumapp}
  \bigoplus_{j=1}^d \ker \mathrm{d}B_j^\oplus(0) = \mathbb{R}^{d(d-2)}.
\end{equation}
We fix $1 \leq j \leq d$ and define $B_j^\oplus :
\mathbb{R}^{d(d-2)} \rightarrow \mathbb{R}^{(d-1)(d-2)}$ as follows.
Let $S^{(j)}$ be the $(d-2)$-tuple obtained by deleting $j$ and
$j+1$\,(mod $d$) from the $d$-tuple $(1,\ldots,d)$. \footnote{There
is some freedom in the choice of the $S^{(j)}$; we only require that
the components of each $S^{(j)}$ are distinct and that for each
fixed $k \in \{1,\ldots,d\}$ there are exactly $d-2$ occurrences of
$k$ over all the components of $S^{(1)},\ldots,S^{(d)}$.} Then
define $B_j^\oplus : \mathbb{R}^{d(d-2)} \rightarrow
\mathbb{R}^{(d-1)(d-2)}$ by
\begin{equation*}
  B_j^\oplus(x) = (B_{S^{(j)}_1}(x),\hdots, B_{S^{(j)}_{d-2}}(x)).
\end{equation*}

To see that \eqref{e:directsumapp} holds, we compute the required
kernels using the fact that
\begin{equation*}
  \ker \mathrm{d}B_j^\oplus(0) = \bigcap_{l = 1}^{d-2} \ker \mathrm{d}B_{S^{(j)}_l}(0)
\end{equation*}
and using straightforward considerations. In order to write these
down we write elements of $\mathbb{R}^{d(d-2)}$ as
\begin{equation*}
  (u_1,u_2,\ldots,u_{d-3},u_{d-2};\widetilde{u}_{d-1})
\end{equation*}
where each $u_j \in \mathbb{R}^{d-1}$ and $\widetilde{u}_{d-1} \in
\mathbb{R}^{d-2}$. Then, using \eqref{e:orthhyp} and
\eqref{e:etadefn}, we have
\begin{align*}
  \ker \mathrm{d}B_1^\oplus (0) & = \{ (u,-u,0,0,\ldots,0,0,0;0) : u \in \langle e_1-e_2 \rangle^\perp \}, \\
  \ker \mathrm{d}B_2^\oplus (0) & = \{ (0,u,-u,0,\ldots,0,0,0;0) : u \in \langle e_2-e_3 \rangle^\perp\}, \\
  \vdots \\
  \ker \mathrm{d}B_{d-3}^\oplus (0) & = \{ (0,0,0,0,\ldots,0,u,-u;0) : u \in \langle e_{d-3}-e_{d-2} \rangle^\perp \}, \\
  \ker \mathrm{d}B_{d-2}^\oplus (0) & = \{ (0,0,0,0,\ldots,0,0,u;(-u_1,\ldots,-u_{d-2})) : u \in \langle e_{d-2}-e_{d-1} \rangle^\perp \}, \\
  \ker \mathrm{d}B_{d-1}^\oplus (0) & = \{ (0,0,0,0,\ldots,0,0,0;\widetilde{u}) : \widetilde{u} \in \mathbb{R}^{d-2}\}, \\
  \ker \mathrm{d}B_d^\oplus (0) & = \{ (u,0,0,0,\ldots,0,0,0;0) : u \in \langle e_1 \rangle^\perp \}.
\end{align*}
An elementary calculation now shows that \eqref{e:directsumapp}
holds.

Consequently, it follows from Theorem \ref{t:main} that there exists
a neighbourhood $U$ of the origin, depending on $d, \beta$ and
$\kappa$, and a constant $C$ depending on $d$, such that
\begin{equation} \label{e:gj}
  \int_U \prod_{j=1}^d g_j(B_j^\oplus(x))\,\mathrm{d}x \leq C\prod_{j=1}^d \|g_j\|_{d-1}
\end{equation}
for all $g_j \in L^{d-1}(\mathbb{R}^{(d-1)(d-2)})$. Now, if
$f_j^\otimes \in L^{d-1}(\mathbb{R}^{(d-1)(d-2)})$ is given by
\begin{equation*}
  f_j^\otimes = \bigotimes_{l = 1}^{d-2} f_{S^{(j)}_l}^{1/(d-2)}
\end{equation*}
then by construction,
\begin{equation*}
   \int_{U} \prod_{j=1}^d f_j^\otimes(B_j^\oplus(x))\,\mathrm{d}x = \int_{U} \prod_{j=1}^d f_j(B_j(x))\,\mathrm{d}x
\end{equation*}
and
\begin{equation*}
  \prod_{j=1}^d \|f_j^\otimes\|_{d-1} = \prod_{j=1}^d \|f_j\|_{(d-1)'}.
\end{equation*}
Thus, \eqref{e:NBL} follows immediately from \eqref{e:gj}.

Finally, by the mean value theorem, it is easy to see that there is
a neighbourhood $V$ of the origin in $(\mathbb{R}^{d-1})^{d-1}$,
depending only on $d, \beta$ and $\kappa$, such that
\begin{equation*}
\int_{V}f_1(u_1)\cdots f_{d-1}(u_{d-1})f_d(u_1+\cdots+u_{d-1})\delta(F(u))\,\mathrm{d}u \leq 2\int_U \prod_{j=1}^d f_j(B_j(x))\,\mathrm{d}x.
\end{equation*}

Hence, whenever $\nabla_{u_j}F(0) = e_j$ and $\|F\|_{C^{1,\beta}}
\leq \kappa$ there exists a neighbourhood $V$ of the origin in
$(\mathbb{R}^{d-1})^{d-1}$, depending only on $d, \beta$ and
$\kappa$, and a constant $C$ depending only on $d$, such that
\begin{equation} \label{e:orthcase}
\int_{V}f_1(u_1)\cdots f_{d-1}(u_{d-1})f_d(u_1+\cdots+u_{d-1})\delta(F(u))\,\mathrm{d}u \leq C \prod_{j=1}^d \|f_j\|_{(d-1)'}
\end{equation}
for all $f_j \in L^{(d-1)'}(\mathbb{R}^{d-1})$.

Now suppose that $F:(\mathbb{R}^{d-1})^{d-1}\rightarrow\mathbb{R}$
is such that $\|F\|_{C^{1,\beta}} \leq \kappa$ and
\begin{equation} \label{e:trans}
|\det(\nabla_{u_1}F(0),\ldots,\nabla_{u_{d-1}}F(0))| > \varepsilon.
\end{equation}
Let $A^\oplus$ be the block diagonal $(d-1)^2 \times (d-1)^2$ matrix
with $d-1$ copies of the matrix
\begin{equation*}
  A = (\nabla_{u_1}F(0),\ldots,\nabla_{u_{d-1}}F(0))^T
\end{equation*}
along the diagonal. Then, by the change of variables $u \mapsto
A^\oplus u$ it follows that
\begin{align*}
& \int_{V}f_1(u_1)\cdots f_{d-1}(u_{d-1})f_d(u_1+\cdots+u_{d-1})\delta(F(u))\,\mathrm{d}u \\
& = |\det(A)|^{-(d-1)} \int_{A^\oplus(V)} \widetilde{f}_1(u_1)\cdots \widetilde{f}_{d-1}(u_{d-1})\widetilde{f}_d(u_1+\cdots+u_{d-1})\delta(\widetilde{F}(u))\,\mathrm{d}u
\end{align*}
where $\widetilde{f}_j = f_j \circ A^{-1}$ and $\widetilde{F} = F
\circ (A^\oplus)^{-1}$. The neighbourhood $V$ of the origin shall be
chosen momentarily.

By \eqref{e:trans} it follows that the norm of $A^{-1}$ is bounded
above by a constant depending on only $d$, $\varepsilon$ and
$\kappa$. It follows that the same conclusion holds for the
$C^{1,\beta}$ norm of $\widetilde{F}$. Since, by construction,
$\nabla_{u_j}\widetilde{F}(0) = e_j$, and by \eqref{e:orthcase}, it
follows that there exists a neighbourhood $V$, depending on only $d,
\beta,\varepsilon$ and $\kappa$, and a constant $C$ depending only
on $d$, such that
\begin{align*}
\int_{A^\oplus(V)} \widetilde{f}_1(u_1)\cdots \widetilde{f}_{d-1}(u_{d-1})\widetilde{f}_d(u_1+\cdots+u_{d-1})\delta(\widetilde{F}(u))\,\mathrm{d}u \leq C\prod_{j=1}^d \|\widetilde{f_j}\|_{(d-1)'}.
\end{align*}
Therefore, by \eqref{e:trans},
\begin{align*}
& \int_{V}f_1(u_1)\cdots f_{d-1}(u_{d-1})f_d(u_1+\cdots+u_{d-1})\delta(F(u))\,\mathrm{d}u \\
& \qquad \qquad \leq C|\det(A)|^{-1/(d-1)} \prod_{j=1}^d \|f_j\|_{(d-1)'} \leq C\varepsilon^{-1/(d-1)} \prod_{j=1}^d \|f_j\|_{(d-1)'}. \\
\end{align*}
This concludes the proof.

\section{Applications to harmonic analysis} \label{section:applications}

\subsection{Multilinear singular convolution inequalities}\label{bhtsect}

Given three transversal and sufficiently regular hypersurfaces in
$\mathbb{R}^3$, the convolution of two $L^2$ functions supported on
the first and second hypersurface, respectively, restricts to a
well-defined $L^2$ function on the third. Under a $C^{1,\beta}$
regularity hypothesis and further scaleable assumptions, this was
proved by Bejenaru, Herr and Tataru in \cite{BHT}.
We note that the inequality underlying this restriction phenomenon
also follows from the nonlinear Loomis--Whitney inequality in
\cite{BCW}; the precise versions of the underlying inequalities
differ in \cite{BHT} and \cite{BCW} because a stronger regularity
assumption is made in \cite{BCW} and a uniform transversality
assumption is made in \cite{BHT}. Here we show that natural higher
dimensional analogues of this phenomenon may be deduced from
Corollary \ref{lemma6}.

For $d\geq 2$ and $1\leq j\leq d$, let $U_j$ be a compact subset of
$\mathbb{R}^{d-1}$ and $\Sigma_j:U_j\rightarrow\mathbb{R}^d$
parametrise a $C^{1,\beta}$ codimension-one submanifold $S_j$ of
$\mathbb{R}^d$.
Let the measure $\mathrm{d}\sigma_j$ on $\mathbb{R}^d$ supported on
$S_j$ be given by
$$
\int_{\mathbb{R}^d}\psi(x)\,\mathrm{d}\sigma_j(x)=\int_{U_j}\psi(\Sigma_j(x'))\,\mathrm{d}x',$$
where $\psi$ denotes an arbitrary Borel measurable function on
$\mathbb{R}^d$.
\begin{theorem}\label{bhthigher}
Suppose that the submanifolds $S_1,\hdots,S_d$ are transversal in a
neighbourhood of the origin, $1 \leq q \leq \infty$ and $p' \leq
(d-1)q'$. Then there exists a constant $C$ such that
\begin{equation}\label{bhtest}
\|f_1\mathrm{d}\sigma_1*\cdots *f_d\mathrm{d}\sigma_d\|_{L^q(\mathbb{R}^d)} \leq C \prod_{j=1}^d\|f_j\|_{L^{p}(\mathrm{d}\sigma_j)}
\end{equation}
for all $f_j\in L^{p}(\mathrm{d}\sigma_j)$ with support in a
sufficiently small neighbourhood of the origin.
\end{theorem}
\begin{remark}
\begin{enumerate}
\item[(i)] By H\"older's inequality it suffices to prove Theorem \ref{bhthigher} when $p' = (d-1)q'$. One can
also verify that the exponents in Theorem \ref{bhthigher} are
optimal, as may be seen by taking $f_j$ to be the characteristic
function of a small cap on $S_j$. As such examples illustrate,
at this level of multilinearity, the transversality hypothesis
prevents any additional curvature hypotheses on the submanifolds
$S_j$ from giving rise to further improvement. See \cite{BCT}
for further discussion of such matters.
\item[(ii)] Certain bilinear versions of Theorem \ref{bhthigher} are well-known and discussed in detail in \cite{T}. In particular, it follows from \cite{TVV} that for transversal $S_1$ and $S_2$ (as above), which are smooth with nonvanishing gaussian curvature, there is a constant $C$ for which
    $$\|f_1\mathrm{d}\sigma_1*f_2\mathrm{d}\sigma_2\|_{L^2(\mathbb{R}^d)} \leq C\|f_1\|_{L^{\frac{4d}{3d-2}}(\mathrm{d}\sigma_1)}\|f_2\|_{L^{\frac{4d}{3d-2}}(\mathrm{d}\sigma_2)}.
    $$
    The exponent $\tfrac{4d}{3d-2}$ here is optimal given the
    $L^2$ norm on the left-hand side. The case $d=3$ of this
    inequality was obtained previously in \cite{MVV}. See for
    instance \cite{Bo2} for earlier manifestations of such
    inequalities.
\item[(iii)]
In particular, when $q=\infty$ inequality \eqref{bhtest} implies
that
\begin{equation*}\label{bhtest0}
f_1\mathrm{d}\sigma_1*\cdots *f_d\mathrm{d}\sigma_d(0)\leq C\prod_{j=1}^d\|f_j\|_{L^{(d-1)'}(\mathrm{d}\sigma_j)}.
\end{equation*}
By duality, this is equivalent to the statement that, provided
$f_1,\hdots,f_{d-1}$ have support restricted to a sufficiently
small fixed neighbourhood of the origin, then the multilinear
operator
$$(f_1,\hdots,f_{d-1})\mapsto f_1\mathrm{d}\sigma_1*\cdots* f_{d-1}\mathrm{d}\sigma_{d-1}\Bigl|_{S_d}$$
is bounded from $L^{(d-1)'}(\mathrm{d}\sigma_1)\times\cdots\times L^{(d-1)'}(\mathrm{d}\sigma_{d-1})$ to
$L^{d-1}(\mathrm{d}\sigma_d)$. For $d=3$ this is a local variant of the result in \cite{BHT}.
\item[(iv)] The proof of Theorem \ref{bhthigher} (below) leads to a stronger uniform statement, whereby
the sizes of the constant $C$ and neighbourhood of the origin may be taken to depend only on natural
transversality and smoothness parameters. We omit the details of this.
\end{enumerate}
\end{remark}

\begin{proof}[Proof of Theorem \ref{bhthigher}]
By multilinear interpolation and the trivial estimate
\begin{equation*}
\|f_1\mathrm{d}\sigma_1*\cdots *f_d\mathrm{d}\sigma_d\|_{L^1(\mathbb{R}^d)}\leq \prod_{j=1}^d\|f_j\|_{L^{1}(\mathrm{d}\sigma_j)},
\end{equation*}
it suffices to prove Theorem \ref{bhthigher} for $q=\infty$.

By considering a rotation in $\mathbb{R}^d$, we may assume without
loss of generality that the submanifolds $S_j$ are hypersurfaces;
i.e. given by $\Sigma_j(x')=(x',\phi_j(x'))$ for $C^{1,\beta}$
functions $\phi_j:U_j\rightarrow \mathbb{R}$. Now, for $f_j$
supported on $S_j$ for each $1\leq j\leq d$, and any
$y\in\mathbb{R}^d$ we may write
\begin{eqnarray*}
\begin{aligned}
&f_1\mathrm{d}\sigma_1*\cdots *f_d\mathrm{d}\sigma_d(y)\\&=
\int_{(\mathbb{R}^d)^d} \prod_{j=1}^d f_j(x_j)\delta(x_{jd}-\phi_j(x_j'))\delta(x_1+\cdots+x_d-y)\,\mathrm{d}x_1\cdots \mathrm{d}x_d\\
&=\int_{U_1\times\cdots\times U_d} \prod_{j=1}^d f_j(x_j',\phi_j(x_j'))\delta(x_1'+\cdots+x_d'-y')\delta(\phi_1(x_1')+\cdots+\phi_d(x_d')-y_d)\,\mathrm{d}x_1'\cdots \mathrm{d}x_d'\\
&=\int_{U_1\times\cdots\times U_d} \prod_{j=1}^{d}g_j(x_j')\delta(x_1'+\cdots+x_d'-y')\delta(\phi_1(x_1')+\cdots+\phi_d(x_d')-y_d)\,\mathrm{d}x_1'\cdots \mathrm{d}x_d'\\
&=\int_{U_1\times\cdots\times U_{d-1}}\prod_{j=1}^{d-1}g_j(x_j')\widetilde{g}_d(x_1'+\cdots+x_{d-1}')\delta(F(x_1',\hdots,x_{d-1}'))\,\mathrm{d}x_1'\cdots \mathrm{d}x_{d-1}'
\end{aligned}
\end{eqnarray*}
where
$$g_j(x_j'):=f_j(x_j',\phi_j(x_j')),\;\;\;\widetilde{g}_d(u):=g_d(y'-u)$$ and
$$F(x_1',\hdots,x_{d-1}')=\phi_1(x_1')+\cdots+\phi_{d-1}(x_{d-1}')+\phi_d(y'-(x_1'+\cdots+x_{d-1}'))-y_d.$$
Observe that $F\in C^{1,\beta}$ uniformly in $y$ belonging to a
sufficiently small neighbourhood of the origin, and that by the
transversality hypothesis (combined with the smoothness hypothesis),
\begin{displaymath}
\det(\nabla_{x_1'}F(0),\cdots,\nabla_{x_{d-1}'}F(0))=\det\left(
\begin{array}{cccc}
1& \cdots & 1 & 1\\
\nabla\phi_{1}(0) &
\cdots & \nabla\phi_{d-1}(0) & \nabla\phi_d(y')\\
\end{array}\right)
\not=0
\end{displaymath}
similarly uniformly. Theorem \ref{bhthigher} now follows by
Corollary \ref{lemma6}.
\end{proof}
Estimates of the type \eqref{bhtest} are intimately related to the
multilinear restriction theory for the Fourier transform, to which
we now turn.

\subsection{A multilinear Fourier extension inequality}

Very much as before, let $U$ be a compact neighbourhood of the
origin in $\mathbb{R}^{d-1}$ and $\Sigma:U\rightarrow\mathbb{R}^d$
parametrise a $C^{1,\beta}$ codimension-one submanifold $S$ of
$\mathbb{R}^d$. To the mapping $\Sigma$ we associate the operator
$\mathcal{E}$, given by
$$
\mathcal{E}g(\xi)=\int_{U}g(x)e^{i\langle \xi,\Sigma(x)\rangle}\,\mathrm{d}x;
$$
here $g\in L^1(U)$ and $\xi\in\mathbb{R}^d$. We note that the formal
adjoint $\mathcal{E}^*$ is given by the restriction
$\mathcal{E}^*f=\widehat{f}\circ\Sigma$, where
\,\,$\widehat{\;}$\,\, denotes the Fourier transform on
$\mathbb{R}^d$. The operator $\mathcal{E}$ is thus referred to as an
\emph{adjoint Fourier restriction operator} or \emph{Fourier
extension operator}.

Suppose that we have $d$ such extension operators
$\mathcal{E}_1,\hdots, \mathcal{E}_d$, associated with mappings
$\Sigma_1:U_1\rightarrow\mathbb{R}^d,\hdots,\Sigma_d:U_d\rightarrow\mathbb{R}^d$
and submanifolds $S_1,\hdots,S_d$. 

\begin{conjecture}[Multilinear Restriction \cite {BCW}, \cite{BCT}]\label{MLRC}
Suppose that the submanifolds $S_1,\hdots, S_d$ are transversal in a
neighbourhood of the origin, $q\geq\tfrac{2d}{d-1}$ and $p'\leq
\tfrac{d-1}{d}q$. Then there exists a constant $C$ for which
\begin{equation}\label{mr}
\Bigl\|\prod_{j=1}^{d}\mathcal{E}_{j}g_{j}\Bigr\|_{L^{q/d}(\mathbb{R}^d)}\leq
C\prod_{j=1}^{d} \|g_{j}\|_{L^{p}(U_j)}
\end{equation}
for all $g_{1},\hdots,g_{d}$ supported in a sufficiently small
neighbourhood of the origin.
\end{conjecture}
\begin{remark}
Conjecture \ref{MLRC} implies Theorem \ref{bhthigher}. To see this
we first observe that for any function $f_j$ on $S_j$,
$\widehat{f_j\mathrm{d}\sigma_j}=\mathcal{E}_jg_j$ where
$g_j=f_j\circ\Sigma_j$. Now, if $2\leq q\leq\infty$ and
$p'=(d-1)q'$, then by the Hausdorff--Young inequality followed by
Conjecture \ref{MLRC},
\begin{align*}
\|f_1\mathrm{d}\sigma_1*\cdots *f_d\mathrm{d}\sigma_d\|_{L^q(\mathbb{R}^d)} & \leq \Bigl\|\prod_{j=1}^d\mathcal{E}_jg_j\Bigr\|_{L^{q'}(\mathbb{R}^d)} \\
& \leq C \prod_{j=1}^d\|g_j\|_{L^{p}(U_j)} = C\prod_{j=1}^d\|f_j\|_{L^{p}(\mathrm{d}\sigma_j)}.
\end{align*}
This link was observed for $d=3$ in \cite{BHT}.
\end{remark}
In \cite{BCT} a local form of Conjecture \ref{MLRC} was proved with
an $\varepsilon$-loss; namely for each $\varepsilon>0$ the above
conjecture was obtained with \eqref{mr} replaced by
\begin{equation}\label{mrk}
\Bigl\|\prod_{j=1}^{d}\mathcal{E}_{j}g_{j}\Bigr\|_{L^{q/d}(B(0,R))}\leq
C_\varepsilon R^\varepsilon\prod_{j=1}^{d} \|g_{j}\|_{L^{p}(U_j)},
\end{equation}
for all $R>0$. In \cite{BCW} the global estimate \eqref{mr} was
obtained for $d=3$ and $q=6$. Here we extend this global result to
all dimensions.
\begin{theorem}\label{mrg}
If $S_1,\hdots, S_d$ are transversal in a neighbourhood of the
origin then there exists a constant $C$ such that
\begin{equation}\label{mrgest}
\Bigl\|\prod_{j=1}^{d}\mathcal{E}_{j}g_{j}\Bigr\|_{L^{2}(\mathbb{R}^d)}\leq C
\prod_{j=1}^{d} \|g_{j}\|_{L^{\frac{2d-2}{2d-3}}(U_j)}
\end{equation}
for all $g_{1},\hdots,g_{d}$ supported in a sufficiently small
neighbourhood of the origin.
\end{theorem}
\begin{proof}
By Plancherel's Theorem, \eqref{mrgest} is equivalent to the
estimate
$$
\|f_1\mathrm{d}\sigma_1*\cdots *f_d\mathrm{d}\sigma_d\|_{L^2(\mathbb{R}^d)}\leq C\prod_{j=1}^d\|f_j\|_{L^{\frac{2d-2}{2d-3}}(\mathrm{d}\sigma_j)},$$ where as before we are identifying $f_j$ with $g_j$ by $g_j=f_j\circ\Sigma_j$.
Theorem \ref{mrg} now follows immediately from Theorem
\ref{bhthigher}.
\end{proof}
\begin{remark}
The Lebesgue exponent $\tfrac{2d-2}{2d-3}$ on the right-hand side of
\eqref{mrgest} is best-possible given the $L^2$ norm on the left.
Again, at this level of multilinearity, the transversality
hypothesis prevents any additional curvature hypotheses from giving
rise to further improvement. See \cite{BCT} for further discussion.
\end{remark}

\appendix
\section{Proposition \ref{p:Finnerorth} implies Proposition \ref{p:linearcase}} \label{appendix:reduction}

Assume that, for each $1 \leq j \leq m$, $B_j : \mathbb{R}^d
\rightarrow \mathbb{R}^{d_j}$ is a linear surjection and
\eqref{e:directsum} holds. Let $\Pi_j : \mathbb{R}^d \rightarrow
\mathbb{R}^{d_j}$ be given by \eqref{e:Pijdefn} where $d_j'$ is the
dimension of $\ker B_j$.

Select \emph{any} set of vectors $\{a_k : k \in \mathcal{K}_j\}$
forming an orthonormal basis for $\ker B_j$; that is, the orthogonal
complement of the subspace spanned by the rows of $B_j$. By
definition of the Hodge star and orthogonality considerations it
follows that
\begin{equation} \label{e:XjBjstaridentity}
 \star X_j(B_j) = \|X_j(B_j)\|_{\Lambda^{d_j}(\mathbb{R}^d)} \bigwedge_{k \in \mathcal{K}_j} a_k.
\end{equation}
Here, $\| \cdot \|_{\Lambda^{d_j}(\mathbb{R}^d)} :
\Lambda^{d_j}(\mathbb{R}^d) \rightarrow [0,\infty)$ is the norm
induced by the standard inner product $\langle \cdot,\cdot
\rangle_{\Lambda^{d_j}(\mathbb{R}^d)} : \Lambda^{d_j}(\mathbb{R}^d)
\times \Lambda^{d_j}(\mathbb{R}^d) \rightarrow \mathbb{R}$ given by
\begin{equation*}
  \langle u_1 \wedge \cdots \wedge u_{d_j},v_1 \wedge \cdots \wedge v_{d_j} \rangle_{\Lambda^{d_j}(\mathbb{R}^d)} = \det( \langle u_k,v_\ell \rangle)_{1 \leq k,\ell \leq d_j}.
\end{equation*}

Let $A$ be the $d \times d$ matrix whose $i$th column is equal to
$a_i$ for each $1 \leq i \leq d$ and let $C_j$ be the $d_j \times
d_j$ matrix given by
\begin{equation*}
C_j = B_jA_j,
\end{equation*}
where $A_j$ is the $d \times d_j$ matrix obtained by deleting from
$A$ the columns $a_k$ for each $k \in \mathcal{K}_j$. Then, by
construction,
\begin{equation*}
\Pi_j = C_j^{-1}B_jA.
\end{equation*}
The matrices $A$ and $C_j$ are invertible by the hypothesis
\eqref{e:directsum}. Using $A$ to change variables one obtains
\begin{equation*} \label{e:A}
\int_{\mathbb{R}^d} \prod_{j=1}^m f_j(B_jx)^{\frac{1}{m-1}} \,\mathrm{d}x = |\det(A)| \int_{\mathbb{R}^d} \prod_{j=1}^m \widetilde{f}_j(\Pi_jx)^{\frac{1}{m-1}} \,\mathrm{d}x,
\end{equation*}
where $\widetilde{f}_j = f_j \circ C_j$, $1 \leq j \leq m$. By
Proposition \ref{p:Finnerorth} it follows that
\begin{eqnarray*}
\int_{\mathbb{R}^d} \prod_{j=1}^m f_j(B_jx)^{\frac{1}{m-1}} \,\mathrm{d}x & \leq & |\det(A)| \prod_{j=1}^m \bigg( \int_{\mathbb{R}^{d_j}} \widetilde{f}_j \,\mathrm{d}x \bigg)^{\frac{1}{m-1}} \\
& = & \frac{|\det(A)|}{\left(\prod_{j=1}^m |\det(C_j)|\right)^{\frac{1}{m-1}}} \prod_{j=1}^m \left( \int_{\mathbb{R}^{d_j}}
f_j \right)^{\frac{1}{m-1}}
\end{eqnarray*}
and it remains to check that
\begin{equation} \label{e:detAdetCstarwedge}
\frac{|\det(A)|}{\left(\prod_{j=1}^m |\det(C_j)|\right)^{\frac{1}{m-1}}} = \left|\star\bigwedge_{j=1}^m\star
X_j(B_j)\right|^{-\frac{1}{m-1}}.
\end{equation}
To this end, note that
\begin{equation*}
\star \bigwedge_{j=1}^m \star X_j(B_j) = \prod_{j=1}^m \|X_j(B_j)\|_{\Lambda^{d_j}(\mathbb{R}^d)} \star \bigwedge_{j=1}^m \bigwedge_{k \in \mathcal{K}_j} a_k
\end{equation*}
by \eqref{e:XjBjstaridentity} and therefore
\begin{equation} \label{e:starwedgeidentity}
  \star \bigwedge_{j=1}^m \star X_j(B_j) = \det(A) \prod_{j=1}^m \|X_j(B_j)\|_{\Lambda^{d_j}(\mathbb{R}^d)}
\end{equation}
since $\mathcal{K}_1,\ldots,\mathcal{K}_m$ partitions
$\{1,\ldots,d\}$.

Again use \eqref{e:XjBjstaridentity} to write
\begin{align*}
  |\det(C_j)| & = \bigg|\bigg\langle X_j(B_j), \bigwedge_{l \notin \mathcal{K}_j} a_{l} \bigg \rangle_{\Lambda^{d_j}(\mathbb{R}^d)} \bigg|\\
  & = \|X_j(B_j)\|_{\Lambda^{d_j}(\mathbb{R}^d)} \bigg|\bigg\langle \star\bigg(\bigwedge_{k \in \mathcal{K}_j} a_k \bigg), \bigwedge_{l \notin \mathcal{K}_j} a_{l} \bigg\rangle_{\Lambda^{d_j}(\mathbb{R}^d)} \bigg|
\end{align*}
and therefore, by definition of the Hodge star,
\begin{equation} \label{e:detCidentity}
  |\det(C_j)| = \|X_j(B_j)\|_{\Lambda^{d_j}(\mathbb{R}^d)} |\det(A)|.
\end{equation}
Now \eqref{e:detAdetCstarwedge} follows from
\eqref{e:starwedgeidentity} and \eqref{e:detCidentity}. This
completes the reduction of Proposition \ref{p:linearcase} to
Proposition \ref{p:Finnerorth}.

\section{A quantitative version of the implicit function theorem} \label{appendix:IFT}

We provide a quantitative version of the implicit function theorem
for $C^{1,\beta}$ functions which we used in the proof of
Proposition \ref{lemma6}.

Below we use the notation $B(0,R)$ to denote the open euclidean ball
centred at the origin with radius $R > 0$ in either $\mathbb{R}^n$
or $\mathbb{R}$; the dimension of the ball will be clear from the
context. Similarly, we denote by $\overline{B}(0,R)$ the closed
euclidean ball centred at the origin with radius $R > 0$.
\begin{theorem} \label{t:IFT}
  Suppose $n \in \mathbb{N}$ and $\beta, \kappa > 0$ are given. Let $R_1,R_2 > 0$ be given by
  \begin{equation} \label{e:R1R2}
    R_1 = \frac{1}{(100\kappa)^{1/\beta}} \min\left\{ 1, \frac{1}{10\kappa}\right\} \quad \text{and} \quad R_2 = \frac{1}{(100\kappa)^{1/\beta}}.
  \end{equation}
  If $F : \mathbb{R}^n \times \mathbb{R} \rightarrow \mathbb{R}$ is such that $\|F\|_{C^{1,\beta}} \leq \kappa$, $F(0,0)=0$ and $\partial_{n+1} F(0,0) = 1$ then there exists a function $\eta : B(0,R_1) \rightarrow \overline{B}(0,R_2)$ such that
  \begin{equation*}
  F(x,\eta(x)) = 0 \quad \text{for each $x$ belonging to $B(0,R_1)$},
  \end{equation*}
  and a constant $\widetilde{\kappa}$, depending on at most $n, \beta$, and $\kappa$, such that $\|\eta\|_{C^{1,\beta}} \leq \widetilde{\kappa}$.
\end{theorem}

\begin{proof} The proof proceeds via a standard fixed point argument applied to
the map $\Psi_x : \overline{B}(0,R_2) \rightarrow \mathbb{R}$ given
by
\begin{equation*}
  \Psi_x(\eta) = \eta - F(x,\eta)
\end{equation*}
for fixed $x \in B(0,R_1)$. We shall prove that $\Psi_x$ is a
contraction which maps $\overline{B}(0,R_2)$ to itself.

Let $\Phi : (\mathbb{R}^n \times \mathbb{R})^2 \rightarrow
\mathbb{R}$ be the map given by
  \begin{equation*}
    \Phi((x_1,\eta_1),(x_2,\eta_2)) = \frac{F(x_2,\eta_2)-F(x_1,\eta_1)-dF(x_1,\eta_1)(x_2-x_1,\eta_2-\eta_1)}{|(x_2-x_1,\eta_2-\eta_1)|}
  \end{equation*}
  whenever $(x_1,\eta_1), (x_2,\eta_2) \in \mathbb{R}^n \times \mathbb{R}$ are distinct, and zero otherwise. By the mean value theorem and the fact that $\|F\|_{C^{1,\beta}} \leq \kappa$ it follows that $\Phi$ is everywhere continuous and
  \begin{equation} \label{e:Phi}
    |\Phi((x_1,\eta_1),(x_2,\eta_2))| \leq 1/4 \quad \text{for all} \quad (x_j,\eta_j) \in \overline{B}(0,R_2) \times \overline{B}(0,R_2).
  \end{equation}
  For each $\eta_1, \eta_2 \in \overline{B}(0,R_2)$ we have
  \begin{equation*}
    \Psi_x(\eta_1) - \Psi_x(\eta_2) = (1- \partial_{n+1}F(x,\eta_1))(\eta_1 - \eta_2) + \Phi((x,\eta_1),(x,\eta_2))|\eta_1-\eta_2|.
  \end{equation*}
  Since $\partial_{n+1}F(0,0) = 1$ and $\|F\|_{C^{1,\beta}} \leq \kappa$ it follows that
  \begin{equation} \label{e:eta}
    |1-\partial_{n+1}F(x,\eta)| \leq 1/4 \quad \text{whenever} \quad (x,\eta) \in \overline{B}(0,R_2) \times \overline{B}(0,R_2).
  \end{equation}
  Hence, by \eqref{e:Phi} and \eqref{e:eta} it follows that
  \begin{equation} \label{e:contraction}
    |\Psi_x(\eta_1) - \Psi_x(\eta_2)| \leq \tfrac{1}{2}|\eta_1-\eta_2|
  \end{equation}
  and $\Psi_x$ is a contraction.

  Now let $\eta \in \overline{B}(0,R_2)$. Using the hypothesis $\|F\|_{C^{1,\beta}} \leq \kappa$, along with \eqref{e:contraction} and \eqref{e:R1R2}, it follows that
  \begin{equation*}
    |\Psi_x(\eta)| \leq |\Psi_x(\eta) - \Psi_x(0)| + |\Psi_x(0)| \leq R_2.
  \end{equation*}
  Hence $\Psi_x(\overline{B}(0,R_2)) \subseteq \overline{B}(0,R_2)$. By the Banach fixed point theorem, there exists a mapping $\eta : B(0,R_1) \rightarrow \overline{B}(0,R_2)$ such that $\Psi_x(\eta(x)) = \eta(x)$, or equivalently $F(x,\eta(x)) = 0$, for each $x \in B(0,R_1)$.

  It remains to show that $\eta$ belongs to $C^{1,\beta}$ and $\|\eta\|_{C^{1,\beta}} \leq \widetilde{\kappa}$ for some constant $\widetilde{\kappa}$ depending on at most $n, \beta$ and $\kappa$. To see that $\eta$ is differentiable, fix $x,h \in B(0,R_1)$ such that $x+h \in B(0,R_1)$. Since $F(x+h,\eta(x+h)) = F(x,\eta(x))$ it follows that
  \begin{equation*}
    dF(x,\eta(x))(h,\eta(x+h)-\eta(x)) + \Phi((x,\eta(x)),(x+h,\eta(x+h)))|(h,\eta(x+h)-\eta(x))| = 0
  \end{equation*}
  and therefore
  \begin{align*}
    & \partial_{n+1}F(x,\eta(x))(\eta(x+h)-\eta(x)) \\ & = - \langle \nabla_x F(x,\eta(x)), h\rangle - \Phi((x,\eta(x)),(x+h,\eta(x+h))|(h,\eta(x+h)-\eta(x))|.
  \end{align*}
  Note that by \eqref{e:Phi} and \eqref{e:eta} it follows that
  \begin{align*} \label{e:etacont}
    |\eta(x+h)-\eta(x)| \leq C|h|
  \end{align*}
  for some finite constant $C$ independent of $h$. Moreover, $\Phi$ is continuous and vanishes along the diagonal. It follows that $\eta$ is differentiable at $x$ and
  \begin{equation*}
    \nabla \eta(x) = - \frac{\nabla_x F(x,\eta(x))}{\partial_{n+1} F (x,\eta(x))}.
  \end{equation*}
  Using $\|F\|_{C^{1,\beta}} \leq \kappa$ and \eqref{e:eta} one quickly obtains the inequality $\|\eta\|_{C^{1,\beta}} \leq \widetilde{\kappa}$ for some constant $\widetilde{\kappa}$ depending only on $n, \beta$ and $\kappa$.
\end{proof}

\end{document}